\documentclass[reqno]{llncs}
\usepackage[utf8]{inputenc}
\usepackage{amsmath}
\usepackage{amsfonts}
\usepackage{amssymb}
\usepackage{enumitem}
\usepackage[svgnames]{xcolor}
\usepackage{colortbl}

\makeatletter
\let\bfseries=\undefined
\DeclareRobustCommand\bfseries
{\not@math@alphabet\bfseries\mathbf
  \boldmath\fontseries\bfdefault\selectfont}
\makeatother

\newcommand{\Z}{{\mathbb Z}}

\newcommand{\R}{\mathbb R}

\DeclareMathOperator{\supp}{supp}

\DeclareMathOperator{\rank}{rank}

\def\ve#1{\mathchoice{\mbox{\boldmath$\displaystyle\bf#1$}}
{\mbox{\boldmath$\textstyle\bf#1$}}
{\mbox{\boldmath$\scriptstyle\bf#1$}}
{\mbox{\boldmath$\scriptscriptstyle\bf#1$}}}

\newcommand\veb{{\ve b}}

\newcommand\ved{{\ve d}}

\newcommand\veg{{\ve g}}

\newcommand\veu{{\ve u}}
\newcommand\vev{{\ve v}}
\newcommand\vew{{\ve w}}
\newcommand\vex{{\ve x}}
\newcommand\vey{{\ve y}}
\newcommand\vez{{\ve z}}

\newcommand{\eoproof}{\hspace*{\fill} $\square$ \vspace{5pt}}

\usepackage{algorithm}
\usepackage{ifthen}
\usepackage{tikz}
\usepackage{subcaption}
\usepackage{here}

\newcommand{\DeclareBracket}[3]{
  \newcommand{#1}[2][]{%
  \ifthenelse%
  {\equal{##1}{}}%
  {\left#2##2\right#3}%
  {\csname ##1l\endcsname#2##2\csname ##1r\endcsname#3}}}
\DeclareBracket\set\{\}

\usepackage[margin=1.5in]{geometry}

\usepackage[colorlinks=true,breaklinks=true,bookmarks=true,urlcolor=blue,
     citecolor=blue,linkcolor=blue,bookmarksopen=false,draft=false]{hyperref}
\usepackage{cleveref}

\newtheorem{innercustomthm}{Theorem}
\newenvironment{customthm}[1]
  {\renewcommand\theinnercustomthm{#1}\innercustomthm}
  {\endinnercustomthm}

\begin{document}

\title{Circuit Walks in Integral Polyhedra}

\author{Steffen Borgwardt\inst{1}\and Charles Viss\inst{2}}

%\thanks{The author was supported by {TopMath}, a graduate program of the {Elite Network of Bavaria} and the {TUM Graduate School}.}

\institute{\email{\href{mailto:steffen.borgwardt@ucdenver.edu}{steffen.borgwardt@ucdenver.edu}};
University of Colorado Denver \and
\email{\href{mailto:charles.viss@ucdenver.edu}{charles.viss@ucdenver.edu}};
University of Colorado Denver 
}

\date{\today}

\maketitle

\begin{abstract}
Circuits play a fundamental role in the theory of linear programming due to their intimate connection to algorithms of combinatorial optimization and the efficiency of the simplex method. We are interested in better understanding the properties of circuit walks in integral polyhedra. In this paper, we introduce a hierarchy for integral polyhedra based on different types of behavior exhibited by their circuit walks. Many problems in combinatorial optimization fall into the most interesting categories of this hierarchy -- steps of circuit walks only stop at integer points, at vertices, or follow actual edges. We classify several classical families of polyhedra within the hierarchy, including $0/1$-polytopes, polyhedra defined by totally unimodular matrices, and more specifically matroid polytopes, transportation polytopes, and partition polytopes. Finally, we prove three characterizations of the simple polytopes that appear in the bottom level of the hierarchy where all circuit walks are edge walks, showing that such polytopes constitute a generalization of simplices and parallelotopes.
\end{abstract}


\section{Introduction}

The search for a polynomial pivot rule for the simplex method is one of the fundamental open questions in linear programming. It motivates the studies of the combinatorial and circuit diameters of polyhedra. The {\em combinatorial diameter} of a polyhedron refers to the maximum number of steps needed to connect any pair of vertices by an edge walk. It is a lower bound on the best-case performance of the simplex method -- in particular, a family of $n$-dimensional polyhedra with $f$ facets whose diameter is super-polynomial in $f$ and $n$ would disprove the existence of a polynomial pivot rule for the simplex algorithm. While this is a classical field of study, there remain many open questions.

One of the attempts to gain a better understanding of the behavior of edge walks is the study of {\em circuit walks} and the associated \textit{circuit diameters}. These generalize the concept of walking along the edges of a polyhedron to walking along its {\em circuits}. Whereas the famous Hirsch Conjecture is false in general \cite{kw-67,s-11}, the analogous \textit{Circuit Diameter Conjecture} \cite{bfh-14}, which asks whether the circuit diameter of a polyhedron is bounded by $f - n$, remains open \cite{bsy-18,sy-15b}.

We introduce some notation \cite{bfh-14,bfh-16}: Given a polyhedron $P = \{ \vex \in \R^n \colon A \vex = \veb, B \vex \leq \ved \}$, \textit{the set of circuits of $P$}, denoted $\mathcal{C}(A,B)$, consists of those $\veg \in \ker(A) \setminus \{ \ve 0 \}$ normalized to coprime integer components for which $B \veg$ is support-minimal over the set $\{B \vex \colon \vex \in \ker(A) \setminus \{\ve0 \} \}$. Circuits also appear as \textit{elementary vectors} in the literature \cite{r-69}. It can be shown that the set of circuits consists of all potential edge directions of $P$ as the right-hand side vectors $\veb$ and $\ved$ vary \cite{g-75}. Note that $\mathcal{C}(A, B)$ is dependent on the representation of a polyhedron. When a polyhedron is not given through an $\mathcal{H}$-representation, we assume that its set of circuits corresponds to that of a minimal representation; i.e., that each constraint appears as a facet.

The directions of $\mathcal{C}(A, B)$ can be used to traverse $P$ via a \textit{circuit walk}:

\begin{definition}[Circuit Walk]\label{def:circuit}
  Let $P=\{ \vex \in \R^n \colon A \vex = \veb, B \vex \leq \ved \}$ be a polyhedron. For two vertices $\vev^{(1)},\vev^{(2)}$ of $P$, we call a sequence $\vev^{(1)}=\vey^{(0)},\ldots,\vey^{(k)}=\vev^{(2)}$ a \textbf{circuit walk} of length $k$ if for $i=0,\ldots,k-1$ we have:
\begin{enumerate}
\item $\vey^{(i)}\in P$,
\item $\vey^{(i+1)}= \vey^{(i)} + \alpha_i\veg^{(i)}$ for some $\veg^{(i)}\in \mathcal{C}(A,B)$  and $\alpha_i>0$, and
\item $\vey^{(i)}+\alpha\veg^{(i)}$ is infeasible for all $\alpha>\alpha_i$. 
\end{enumerate}
If $\vey^{(i)}$ is a vertex of $P$ for $i=0,\ldots,k$, we call the circuit walk a \textbf{vertex walk}. If $\vey^{(i)}$ has integer components for $i=0,\ldots,k$, we call the circuit walk \textbf{integral} (and \textbf{non-integral} otherwise).
\end{definition}

\noindent Informally, circuit walks travel from an initial vertex to a terminating vertex by following circuit directions and taking steps of maximal length. In particular, these steps may go through the interior of $P$.

Further, as a generalization of the edge directions, $\mathcal{C}(A,B)$ provides an optimality certificate for any linear program over $P$ \cite{g-75}. Thus, there are many settings in mathematical programming in which algorithms construct circuit walks by taking augmenting steps along circuits \cite{bk-84,blswz-99,b-77,dhl-15,ft-97,gdl-14,how-11}. For example, the computation of an improving circuit direction is a viable approach for dealing with highly degenerate vertices in the simplex method \cite{gdl-14}. Additionally, an augmentation scheme along so-called \textit{greedy} circuit directions takes only polynomially many steps \cite{dhl-15,how-11}. The challenge here lies in finding a greedy circuit direction -- it is open whether this can be done in polynomial time. However, it is possible to efficiently compute circuits for a \textit{steepest-descent} augmentation scheme \cite{dhl-15}, which terminates in at most $|\mathcal{C}(A, B)|$ steps and runs in strongly polynomial time for polyhedra defined by totally unimodular matrices \cite{bv-18}.

We are interested in the behavior of circuit walks constructed by such algorithms. We are especially interested in circuit walks within \textit{integral polyhedra} -- those polyhedra whose vertices have integer coordinates -- due to the intimate relationship between circuits and methods from combinatorial optimization \cite{b-13,bdfm-18,bh-17,kps-17}: Many algorithms for classical problems from combinatorial optimization, such as transportation or network flow problems, can be interpreted as circuit walks in the underlying polyhedra. For example, in the context of a minimum-cost flow problem, circuits correspond to directed cycles in the associated network. It follows that the minimum mean cycle-canceling algorithm serves as an efficient implementation of the steepest-descent circuit augmentation scheme \cite{dhl-15,gdl-14}. Graph-theoretic interpretations of circuits can also be used to prove bounds on the circuit diameter of classical polyhedra such as matching polytopes and the traveling salesman polytope \cite{kps-17}.

By definition, a circuit walk in an integral polyhedron begins and ends at vertices. In general, however, the intermediate steps of the walk need not travel along edges, terminate at vertices, or even visit integral points. In this paper, we therefore define a hierarchy of integral polyhedra based on which of these various behaviors appear in a polyhedron. We contrast our approach with that of \cite{bdf-16} in which different relaxations of \Cref{def:circuit} (such as circuit walks that do not take maximal steps or that may even leave the polyhedron) are used to define a hierarchy of circuit diameters. Our main results regarding the proposed hierarchy are presented in \Cref{sec:results}.

\subsection*{Outline}

First, in \Cref{sec:hierarchy} we formally define the hierarchy and show that all of its levels are distinct (\Cref{thm:hierarchy}). We also relate the hierarchy to an important challenge in the study of circuit diameters: all circuit walks in a polyhedron are vertex walks if and only if all of its circuit walks are reversible (\Cref{thm:reversible}). Short proofs are provided in \Cref{sec:proofsforthms}.

Next, in \Cref{sec:specialpolyhedra} we discuss the relationship between the hierarchy and two well-known families of integral polyhedra: $0/1$-polytopes and those defined by totally unimodular matrices (or \textit{TU polyhedra}). We show that circuit walks in general $0/1$-polytopes can exhibit any behavior. However, we prove that all circuit walks in TU polyhedra are integral (\Cref{thm:tu_circuits}), which implies that all circuit walks in TU $0/1$-polytopes are vertex walks (\Cref{cor:tu_0/1-polytopes}). We classify examples of these polyhedra within the hierarchy (\Cref{thm:examples}): matroid polytopes, transportation polyhedra, and so-called \textit{bounded-size} and \textit{fixed-size} partition polytopes. These results, which include a characterization of the edges and circuits of the bounded-size partition polytope, are proven in \Cref{sec:proofofthm1,sec:proof_thm_examples}.  

Finally, in \Cref{sec:edges=circuits} we provide several characterizations of the simple polytopes in which all circuit walks are necessarily edge walks (\Cref{thm:inner_and_elem_cones,thm:opposite_inner_cones,thm:non-degerate_edges=circuits}). We show that such polytopes constitute a highly-symmetric generalization of the simplex and the $n$-parallelotope which we call the \textit{$(n, d)$-parallelotope.} Proofs for these results are given in \Cref{sec:proof_thms_ecw}. 

\section{Results}\label{sec:results}

\subsection{A Hierarchy of Integral Polyhedra}\label{sec:hierarchy}
We begin with the introduction of a hierarchy for integral polyhedra based on the behavior of their circuit walks. In particular, we classify a polyhedron according to the types of intermediate points which are reachable via circuit walks. See \Cref{fig:hierarchy_reversible} for a visualization of the hierarchy.

The levels of the hierarchy are successively more restrictive. At the top, least-restrictive level of the hierarchy are integral polyhedra with general circuit walk (\textit{GCW}) behavior -- namely, their circuit walks may be non-integral. Below this are \textit{ICW polyhedra}: integral polyhedra in which all circuit walks are necessarily integral. This is followed by \textit{VCW polyhedra}: integral polyhedra in which all circuit walks are vertex walks. The bottom, most-restrictive level consists of \textit{ECW polyhedra} in which the only circuit walks are edge walks. Low-dimensional examples of polyhedra from each level of the hierarchy are given in \Cref{fig:examples}. The behaviors of the circuit walks in these polyhedra yield our first result, formally proven in \Cref{sec:proof_thm_hierarchy}.

\begin{theorem}\label{thm:hierarchy}
All integral polyhedra fall into the hierarchy based on circuit walk behavior depicted in \Cref{fig:hierarchy_reversible}. The four levels of the hierarchy -- GCW, ICW, VCW, and ECW polyhedra -- are distinct.
\end{theorem}

\definecolor{polytopeColor}{gray}{0.95}
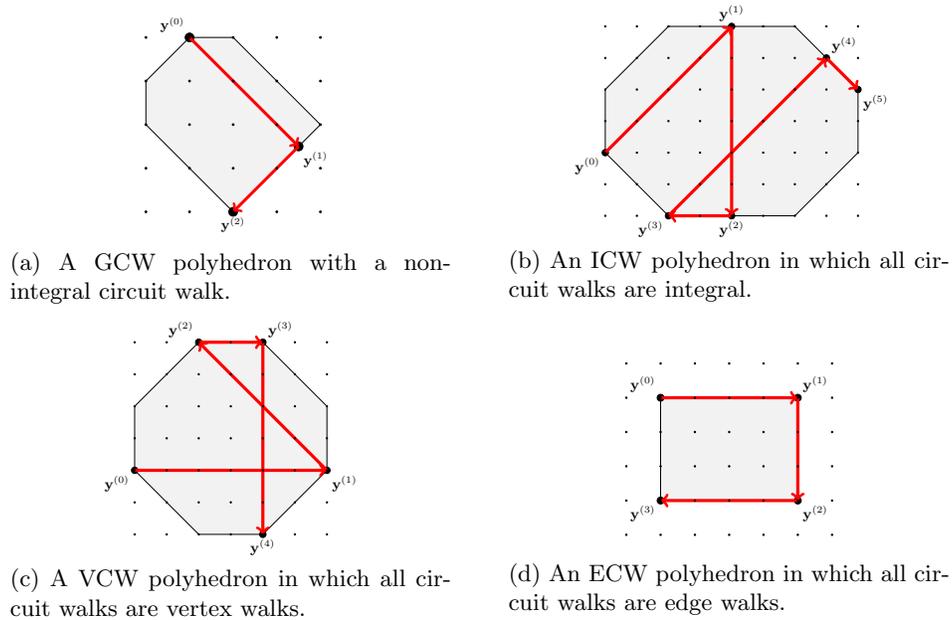
\begin{figure}[t]
    \centering
    \begin{subfigure}{0.42 \textwidth}
    \centering
    \scalebox{.58}{
			\begin{tikzpicture}[scale=1]
			\useasboundingbox (-1,-2.5) rectangle (5,2.3) ;
					\coordinate (v1) at (0,1);
					\coordinate (v2) at (1,2);
					\coordinate (v3) at (2,2);
					\coordinate (v4) at (4,0);
					\coordinate (v5) at (2,-2);
					\coordinate (v6) at (0,0);
					\coordinate (y1) at (3.5, -0.5);
					\draw[black, fill= polytopeColor] (v1)--(v2)--(v3)--(v4)--(v5)--(v6)--(v1);							
					\draw [fill, black] (v2) circle [radius=0.1];
					\draw [fill, black] (v5) circle [radius=0.1];
					\draw [fill, black] (y1) circle [radius=0.1];
					\node[above left] at (v2) {$\vey^{(0)}$};
					\node[below] at (v5) {$\vey^{(2)}$};
					\node[below right] at (y1) {$\vey^{(1)}$};
					\draw[draw=red, line width= 2,  ->] (v2)--(y1);
					\draw[draw=red, line width= 2,  ->] (y1)--(v5);
					\foreach \x in {0,1,...,4}{
        		\foreach \y in {-2,-1,...,2}{
				 			\draw [fill, black] (\x,\y) circle [radius=0.02];
						}
					}
			\end{tikzpicture}}
			\caption{A GCW polyhedron with a non-integral circuit walk.}\label{fig:example_a}
    \end{subfigure}
    \qquad
    \begin{subfigure}{0.42 \textwidth}
    \centering
    \scalebox{.6}{
    \begin{tikzpicture}[scale=.7]
    \useasboundingbox (0,-3.5) rectangle (8,3) ;
					\coordinate (v1) at (0,1);
					\coordinate (v2) at (2,3);
					\coordinate (v3) at (6,3);
					\coordinate (v4) at (8,1);
					\coordinate (v5) at (8,-1);
					\coordinate (v6) at (6,-3);
					\coordinate (v7) at (2,-3);
					\coordinate (v8) at (0,-1);
					\coordinate (y1) at (4,3);
					\coordinate (y2) at (4,-3);
					\coordinate (y4) at (7,2);
					\draw[black, fill= polytopeColor] (v1)--(v2)--(v3)--(v4)--(v5)--(v6)--(v7)--(v8)--(v1);							
					\draw [fill, black] (v8) circle [radius=0.1];
					\draw [fill, black] (y1) circle [radius=0.1];
					\draw [fill, black] (y2) circle [radius=0.1];
					\draw [fill, black] (v7) circle [radius=0.1];
					\draw [fill, black] (y4) circle [radius=0.1];
					\draw [fill, black] (v4) circle [radius=0.1];
					\node[below left] at (v8) {$\vey^{(0)}$};
					\node[above] at (y1) {$\vey^{(1)}$};
					\node[below] at (y2) {$\vey^{(2)}$};
					\node[below left] at (v7) {$\vey^{(3)}$};
					\node[above right] at (y4) {$\vey^{(4)}$};
					\node[below right] at (v4) {$\vey^{(5)}$};
					\draw[draw=red, line width= 2,  ->] (v8)--(y1);
					\draw[draw=red, line width= 2,  ->] (y1)--(y2);
					\draw[draw=red, line width= 2,  ->] (y2)--(v7);
					\draw[draw=red, line width= 2,  ->] (v7)--(y4);
					\draw[draw=red, line width= 2,  ->] (y4)--(v4);
					\foreach \x in {0,1,...,8}{
        		\foreach \y in {-3,-2,...,3}{
				 			\draw [fill, black] (\x,\y) circle [radius=0.02];
						}
					}
			\end{tikzpicture}}
			\caption{An ICW polyhedron in which all circuit walks are integral.}\label{fig:example_b}
    \end{subfigure}
    
    \begin{subfigure}{0.42 \textwidth}
    \centering
    \scalebox{.6}{
			\begin{tikzpicture}[scale=.71]
            \useasboundingbox (0,-3.5) rectangle (6,4) ;
					\coordinate (v1) at (0,1);
					\coordinate (v2) at (2,3);
					\coordinate (v3) at (4,3);
					\coordinate (v4) at (6,1);
					\coordinate (v5) at (6,-1);
					\coordinate (v6) at (4,-3);
					\coordinate (v7) at (2,-3);
					\coordinate (v8) at (0,-1);
					\draw[black, fill= polytopeColor] (v1)--(v2)--(v3)--(v4)--(v5)--(v6)--(v7)--(v8)--(v1);							
					\draw [fill, black] (v8) circle [radius=0.1];
					\draw [fill, black] (v5) circle [radius=0.1];
					\draw [fill, black] (v2) circle [radius=0.1];
					\draw [fill, black] (v3) circle [radius=0.1];
					\draw [fill, black] (v6) circle [radius=0.1];
					\node[below left] at (v8) {$\vey^{(0)}$};
					\node[below right] at (v5) {$\vey^{(1)}$};
					\node[above left] at (v2) {$\vey^{(2)}$};
					\node[above right] at (v3) {$\vey^{(3)}$};
					\node[below] at (v6) {$\vey^{(4)}$};
					\draw[draw=red, line width= 2,  ->] (v8)--(v5);
					\draw[draw=red, line width= 2,  ->] (v5)--(v2);
					\draw[draw=red, line width= 2,  ->] (v2)--(v3);
					\draw[draw=red, line width= 2,  ->] (v3)--(v6);
					\foreach \x in {0,1,...,6}{
        		\foreach \y in {-3,-2,...,3}{
				 			\draw [fill, black] (\x,\y) circle [radius=0.02];
						}
					}
			\end{tikzpicture} }
			\caption{A VCW polyhedron in which all circuit walks are vertex walks.}\label{fig:example_c}
    \end{subfigure}
    \qquad
    \begin{subfigure}{0.42 \textwidth}
    \centering
    \scalebox{.6}{
			\begin{tikzpicture}[scale=.76]
            \useasboundingbox (0,-2.3) rectangle (6,4.4) ;
					\coordinate (v1) at (1,2);
					\coordinate (v2) at (5,2);
					\coordinate (v3) at (5,-1);
					\coordinate (v4) at (1,-1);
					\draw[black, fill= polytopeColor] (v1)--(v2)--(v3)--(v4)--(v1);							
					\draw [fill, black] (v1) circle [radius=0.1];
					\draw [fill, black] (v2) circle [radius=0.1];
					\draw [fill, black] (v3) circle [radius=0.1];
					\draw [fill, black] (v4) circle [radius=0.1];
					\node[above left] at (v1) {$\vey^{(0)}$};
					\node[above right] at (v2) {$\vey^{(1)}$};
					\node[below right] at (v3) {$\vey^{(2)}$};
					\node[below left] at (v4) {$\vey^{(3)}$};
					\draw[draw=red, line width= 2,  ->] (v1)--(v2);
					\draw[draw=red, line width= 2,  ->] (v2)--(v3);
					\draw[draw=red, line width= 2,  ->] (v3)--(v4);
					\foreach \x in {0,1,...,6}{
        		\foreach \y in {-2,-1,...,3}{
				 			\draw [fill, black] (\x,\y) circle [radius=0.02];
						}
					}
			\end{tikzpicture} }
\caption{An ECW polyhedron in which all circuit walks are edge walks.}\label{fig:example_d}
    \end{subfigure}
    \caption{Two-dimensional examples of polyhedra from each level of the hierarchy.}
    \label{fig:examples}
\end{figure}

\begin{figure}[b]
\begin{center}
\includegraphics[width=.6\textwidth]{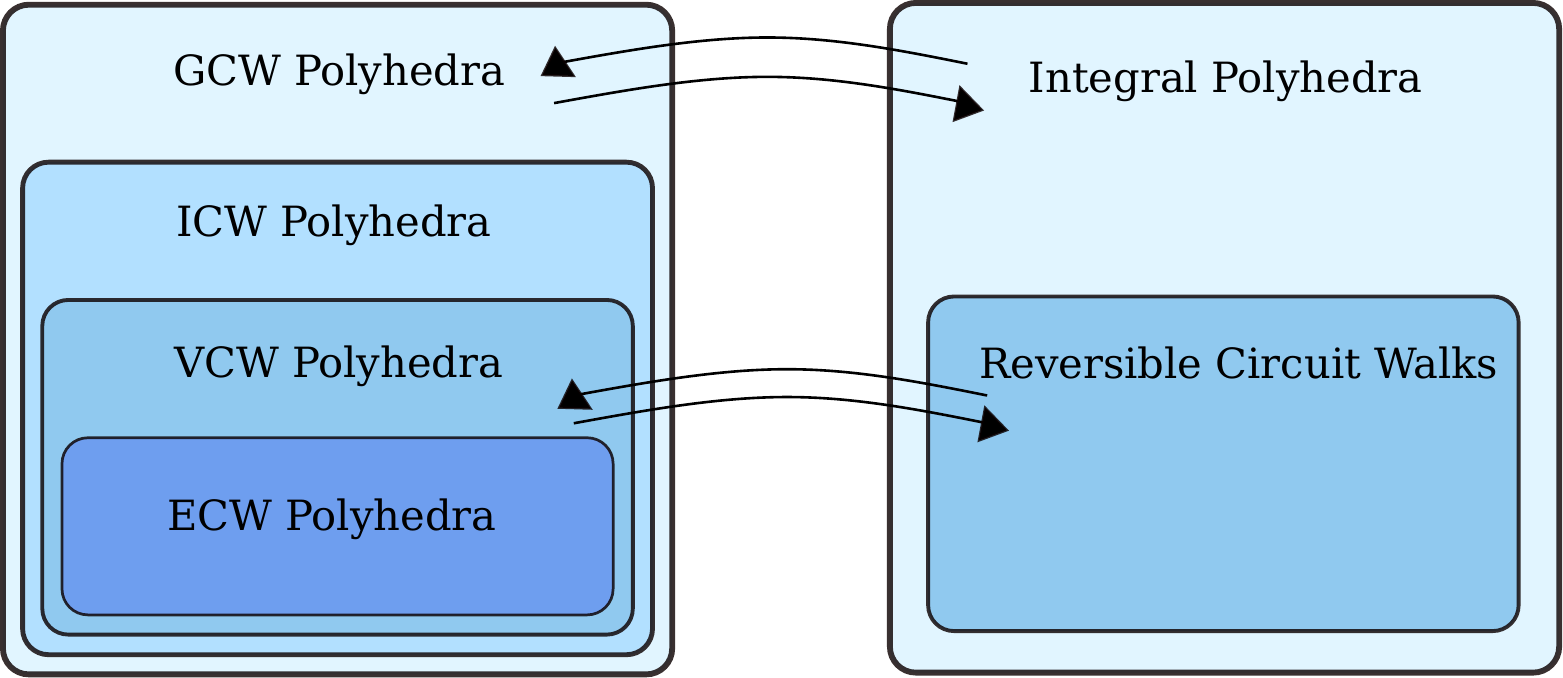}
\end{center}
\caption{A hierarchy for integral polyhedra based on the behavior of their circuit walks.  The VCW level is equivalent to the polyhedra in which all circuit walks are reversible.}\label{fig:hierarchy_reversible}
\end{figure}

The location of a polyhedron within this hierarchy has strong implications on the behavior of a circuit augmentation scheme when applied to the polyhedron. The middle levels of the hierarchy ensure that the algorithm will only visit either integral points or vertices of the polyhedron. At the bottom level, any circuit augmentation scheme is necessarily some variation of the simplex method.

We can also connect this hierarchy to one of the biggest challenges in the study of circuit diameters: Unlike edge walks, circuit walks are not necessarily reversible \cite{bfh-14}. For example, reversing the walks depicted in \Cref{fig:example_a,fig:example_b} does not yield maximal circuit walks. However, as depicted in \Cref{fig:hierarchy_reversible} and proven in \Cref{sec:reversible}, the hierarchy characterizes the polyhedra in which all circuit walks are reversible:

\begin{theorem}\label{thm:reversible}
All circuit walks in a polyhedron are reversible if and only if all circuit walks in the polyhedron are vertex walks.
\end{theorem}

\subsection{0/1-polytopes and Totally Unimodular Matrices}\label{sec:specialpolyhedra}

Next, we discuss how two important families of polyhedra from combinatorial optimization relate to the hierarchy: \textit{0/1-polytopes} and \textit{TU polyhedra}.

0/1-polytopes are widely studied due to their relationship to classical combinatorial optimization problems involving binary decisions. Their combinatorial diameter satisfies the Hirsch Conjecture \cite{n-89} and hence also the Circuit Diameter Conjecture.

We note that in general, circuit walks in 0/1-polytopes need not be integral. Consider the example in \Cref{fig:0_1_example}. The edge direction $(1,1,0)$ is a circuit. However, when taking a step in this direction starting at the vertex $(0,0,0)$, we reach the non-integral midpoint $(\frac{1}{2}, \frac{1}{2}, 0)$ of an edge. 

\begin{figure}[b]
\begin{center}
\begin{tikzpicture}
[scale=2.0, vertices/.style={draw, fill=black, circle, inner sep=0.5pt}]

\filldraw[fill=black!05, draw=black] (0,0)--(1,0)--(1.33,1.33)--(0,1)--cycle;

\node[vertices, label=below left:{$(0,0,0)$}] (a) at (0,0) {};
\node[vertices, label=below right:{$(1,0,0)$}] (b) at (1,0) {};
\node[vertices, label=right:{$(1,1,1)$}] (c) at (1.33,1.33) {};
\node[vertices, label=above left:{$(0,0,1)$}] (d) at (0,1) {};
\node[vertices, label=left:{$(0,1,0) \ $}] (e) at (.33,.33) {};
\node[vertices] (f) at (.61,.2) {};

\fill [fill opacity=0.2,fill=black!05!black] (0,1) -- (1.33,1.33) -- (.33,.33) -- cycle;

\fill [fill opacity=0.1,fill=black!05!black] (1,0) -- (1.33,1.33) -- (0.33,0.33) -- cycle;

\fill [fill opacity=0.04,fill=black!05!black] (0,0) -- (.33,.33) -- (0,1) -- cycle;

\fill [fill opacity=0.07,fill=black!05!black] (1,0) -- (1.33,1.33) -- (0,1) -- cycle;

\foreach \to/\from in {a/b,b/c,c/d,d/a,b/d}
	\draw [line width = 1.1, -] (\to)--(\from);

\foreach \to/\from in {e/a,e/b,e/c,e/d}
	\draw [dashed, opacity = 0.3, line width = 1.0] (\to)--(\from);
	
\draw [fill, black] (a) circle [radius=0.03];
\draw [fill, black] (b) circle [radius=0.03];
\draw [fill, black] (c) circle [radius=0.03];
\draw [fill, black] (d) circle [radius=0.03];
\draw [fill, black] (e) circle [radius=0.025];

\draw[draw=red, line width= 2,  ->] (a)--(f);

\end{tikzpicture}
\caption{A 0/1-polytope in $\R^3$ with non-integral circuit walks.}\label{fig:0_1_example}
\end{center}
\end{figure}
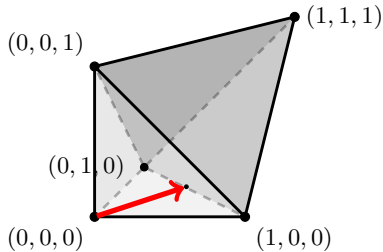

There are important classes of 0/1-polytopes from combinatorial optimization which also exhibit GCW behavior. For example, \textit{matroid polytopes} are a family of 0/1-polytopes in which circuit walks need not be integral. Given a matroid $M$ with ground set $E$ and rank function $f$, the matroid polytope $P(f)$ associated with $M$ is the convex hull of the incidence vectors $\vex \in \{0,1\}^E$ of the independent sets of $M$. We use a common inequality representation \cite{t-84} to describe the possible behaviors of the circuit walks in $P(f)$ in \Cref{sec:matroid}.

Another important class of integral polyhedra are those defined by totally unimodular matrices -- which we call \textit{TU polyhedra}. These polyhedra are guaranteed to be integral for any integral right-hand side, and are extensively studied due to appearing in transportation, assignment, and network flow problems. Their combinatorial (and circuit) diameter is polynomially bounded \cite{bseh-12,df-94} and linear programming over them is efficient \cite{t-86}.

We show that all TU polyhedra are in fact ICW. This suggests that any algorithm which traverses such a polyhedron via a circuit walk has a combinatorial interpretation.

\begin{theorem}\label{thm:tu_circuits}
Let $P = \{ \vex \in \R^n \colon A \vex = \veb, B \vex \leq \ved \}$ be an integral polyhedron whose constraint matrix $\binom{A}{B}$ is totally unimodular. Then all circuit walks in $P$ are integral.
\end{theorem}

We provide a proof of \Cref{thm:tu_circuits} in \Cref{sec:proofofthm1}, but we see another good way to prove the claim from a careful extension of Proposition 3.3 in \cite{o-10}. The merit of our approach is that it provides a generalization to polyhedra in any representation and requires only basic linear algebra.

Note that total unimodularity is not a necessary condition for the ICW property. This can be observed in the polyhedra from \Cref{fig:example_b,fig:example_c}, which can be represented via non-TU constraint matrices. Additionally, total unimodularity is not a sufficient condition for the VCW property. For instance, the \textit{transportation polytope}, whose circuits are characterized in \cite{bdfm-18}, can have many non-vertex circuit walks. We discuss the behavior of these walks in \Cref{sec:tp}.

However, by combining \Cref{thm:tu_circuits} with the fact that all integral points in 0/1-polytopes are vertices, we immediately see that all 0/1-polytopes defined by TU matrices are VCW. 
\begin{corollary}\label{cor:tu_0/1-polytopes}
Let $P = \{ \vex \in \R^n \colon A \vex = \veb, B \vex \leq \ved \}$ be an integral polyhedron whose constraint matrix $\binom{A}{B}$ is totally unimodular. If $P$ is a 0/1-polytope, then all circuit walks in $P$ are vertex walks.
\end{corollary}

One example of such a polytope is the \textit{bounded-size partition polytope} $PP(\kappa^{\pm})$, which is associated with the partitioning of a set of $n$ items into $k$ clusters $C_1,...,C_k$ where each cluster $C_i$ has an upper bound $\kappa_i^+$ and a lower bound $\kappa_i^-$ on its size. The vertices of $PP(\kappa^{\pm})$ correspond to all feasible clusterings. In \Cref{sec:bounded_size}, we characterize the edges and the circuits of this polytope. It then follows that although both circuit walks and edge walks in $PP(\kappa^{\pm})$ travel from vertex to vertex (and hence from clustering to clustering), the circuit walks exhibit a more general behavior.

On the other hand, the related \textit{fixed-size partition polytope} $PP(\kappa)$ is a TU 0/1-polytope associated with the partitioning of a set of $n$ items into $k$ clusters $C_1,...,C_k$ where the size of cluster $C_i$ is fixed at $\kappa_i$ \cite{b-13}. In \Cref{sec:fixedpp}, we show that, unlike the bounded-size partition polytope, all circuit walks in $PP(\kappa)$ are in fact edge walks.

The examples from this section together imply an additional result. Namely, as proven in \Cref{sec:proof_thm_examples}, there exist well-known polyhedra from combinatorial optimization in each level of the hierarchy from \Cref{sec:hierarchy}.

\begin{theorem}\label{thm:examples}
There exist specific examples of integral polyhedra from combinatorial optimization in each of the four distinct levels of the hierarchy based on circuit walk behavior. In particular:
\begin{enumerate}[label=\alph*)]
    \item There exist matroid polytopes which are GCW but not ICW.
    \item There exist transportation polytopes which are ICW but not VCW.
    \item There exist bounded-size partition polytopes which are VCW but not ECW.
    \item All fixed-size partition polytopes are ECW.
\end{enumerate}
\end{theorem}

\begin{figure}[b]
\begin{center}
\includegraphics[width=.95\textwidth]{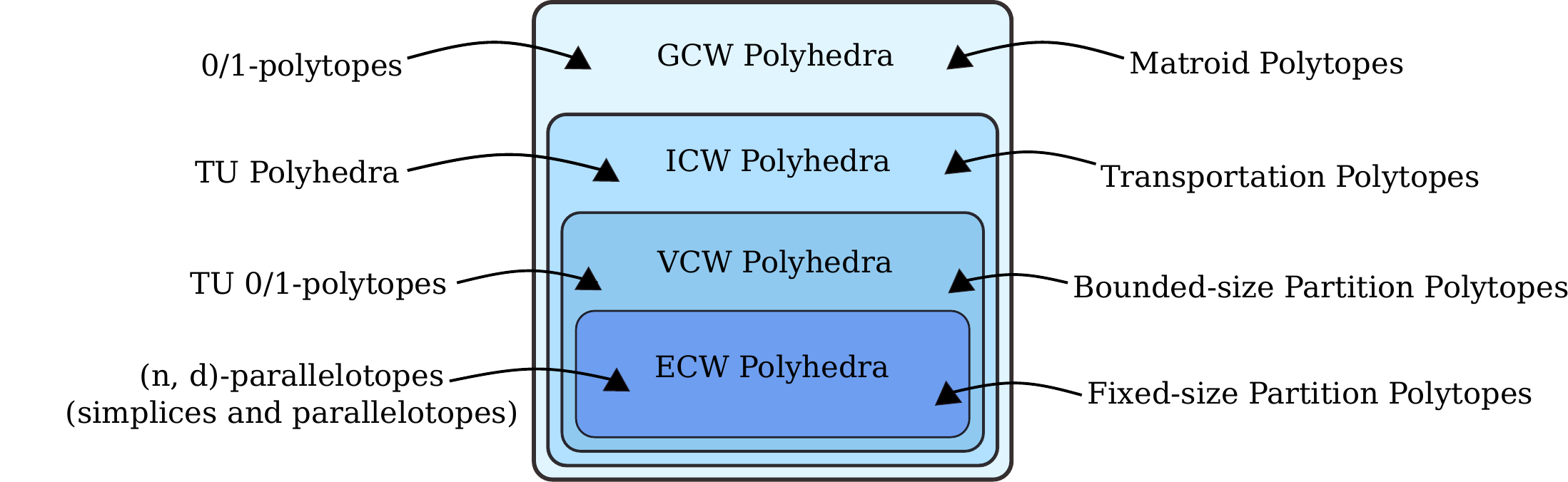}
\end{center}
\caption{The results of \Cref{sec:specialpolyhedra,sec:edges=circuits} in terms of the hierarchy from \Cref{sec:hierarchy}. The left-hand side shows where well-known families of polyhedra belong in the hierarchy, and the right-hand side gives specific examples from combinatorial optimization for each level.}\label{fig:hierarchy_results}
\end{figure}

A visualization of the results from this section (along with the results from the upcoming \Cref{sec:edges=circuits}) in terms of the hierarchy is given in \Cref{fig:hierarchy_results}.

\subsection{ECW Polyhedra} \label{sec:edges=circuits}

In this section we characterize the simple polytopes at the bottom level of the hierarchy in which all circuit walks are edge walks. Our characterizations are quite restrictive and have strong implications on the degenerate polytopes in the category as well. But surprisingly, non-trivial examples of (degenerate) polyhedra from this category do appear in practice. We provide an example -- the \textit{fixed-size partition polytope} -- at the end of the section. 

An $n$-dimensional polyhedron is said to be \textit{simple} -- or \textit{non-degenerate} -- if each vertex belongs to exactly $n$ facets. On the other hand, a \textit{degenerate} polyhedron contains a vertex belonging to more than $n$ facets. Simple polyhedra are of interest in the study of diameters as it suffices to only consider this class of polyhedra to bound the combinatorial diameter of any $n$-dimensional polyhedron with a fixed number of facets \cite{kw-67}. While much harder to prove, the same holds for circuit diameters \cite{bsy-18}. 

The structure of simple polyhedra offers several useful characterizations of ECW polytopes. Our main result is that the only simple polytopes exhibiting this behavior are intimately related to the highly-structured simplex and parallelotope. %(the two examples from \Cref{fig:example_d}.)

Our first important tool is what we call \textit{elementary cones}. Consider a full-dimensional polyhedron $P = \{ \vex \in \R^n \colon B \vex \leq \ved \}$ and the hyperplane arrangement in $\R^n$ consisting of the hyperplanes $B_i \vex = 0$ for each row $B_i$ of $B$. These hyperplanes each contain the origin and partition $\R^n$ into $n$-dimensional polyhedral cones with disjoint interiors. We call this arrangement of hyperplanes the {\em elementary arrangement} of $P$ and refer to the inclusion-minimal $n$-dimensional cones in the arrangement as the \textit{elementary cones} of $P$. 

It is not difficult to see that elementary cones are generated by circuits (\Cref{lem:cone_circuits}). We use this fact to give a first characterization of simple ECW polytopes: the inner cones of all vertices must be elementary cones. Recall that given a vertex $\vev$ of a polyhedron $P$, the \textit{inner cone} $I(\vev)$ of $\vev$ is the cone consisting of \textit{all} feasible directions at $\vev$ with respect to $P$.

\begin{theorem}[Elementary Cone Condition]\label{thm:inner_and_elem_cones}
Let $P = \{ \vex \in \R^n \colon B \vex \leq \ved \}$ be a full-dimensional, simple polytope. All circuit walks in $P$ are edge walks if and only if for each vertex $\vev \in P$, the inner cone $I(\vev)$ is an elementary cone of $P$.
\end{theorem}

Next, we prove that all polytopes satisfying this \textit{elementary cone condition} are highly symmetric: the inner cones of vertices that do not share a facet are opposites of each other. By imposing that this property transfers to vertices belonging to a common face (and stating it with respect to the affine hull of the face), we obtain a second characterization. Given a pair of vertices $\veu, \vev$ of a polyhedron $P$, we let $P^{uv}$ denote the inclusion-minimal face of $P$ containing $\veu$ and $\vev$ and let $I^{uv}(\veu), I^{uv}(\vev)$ denote the inner cones of $\veu, \vev$ with respect to $P^{uv}$. 

\begin{theorem}[Symmetric Inner Cone Condition]\label{thm:opposite_inner_cones}
Let $P$ be a simple polytope given by a minimal representation. All circuit walks in $P$ are edge walks if and only if $I^{uv}(\veu) =  -I^{uv}(\vev)$ for all pairs of vertices $\veu, \vev$ in $P$.
\end{theorem}

Naturally, the \textit{symmetric inner cone condition} of \Cref{thm:opposite_inner_cones} is only satisfied by quite symmetric polyhedra. One such example is the \textit{parallelotope} of dimension $n$, also called the \textit{$n$-parallelotope}: a vertex-transitive polytope with $n$ pairs of parallel, opposite facets and $2^n$ vertices. Equivalently, an $n$-parallelotope is a zonotope generated by a set of $n$ linearly independent vectors in general position, which correspond to its edge directions. We show that the only other polytopes satisfying the symmetric inner cone condition constitute a generalization of the $n$-parallelotope which we call  the \textit{$(n,d)$-parallelotope}.

\begin{definition}[$(n,d)$-parallelotope]\label{def:n_d-parallelotope}
Given $d \in \{1,...,n\}$, an \textit{$(n,d)$-parallelotope} is an $n$-dimensional polytope with $n+d$ facets which satisfies the symmetric inner cone condition and in which each vertex belongs to a $d$-parallelotope face.
\end{definition}

Note that an $(n,n)$-parallelotope is simply an $n$-parallelotope and an $(n, 1)$-parallelotope is an $n$-simplex. Other instances of $(n,d)$-parallelotopes, such as the $(3,2)$-parallelotope pictured in \Cref{fig:polytope_example}, are highly symmetric hybrids of the simplex and the parallelotope. We show that all $(n,d)$-parallelotopes are ECW (\Cref{lem:edge_walks}), which yields a final characterization.
\begin{theorem}\label{thm:non-degerate_edges=circuits}
Let $P$ be an $n$-dimensional, simple polytope given by a minimal representation. All circuit walks in $P$ are edge walks if and only if $P$ is an $(n,d)$-parallelotope.
\end{theorem}
\noindent The proofs of \Cref{thm:inner_and_elem_cones,thm:opposite_inner_cones,thm:non-degerate_edges=circuits} and the related \Cref{lem:cone_circuits,lem:opposite_inner_cones,lem:parallelotope,lem:n-k_parallelotope,lem:edge_walks} in \Cref{sec:proof_thms_ecw} provide further insight into the structure of $(n,d)$-parallelotopes.

\begin{figure}[t]
\begin{center}
\begin{tikzpicture}
[scale=2.2, vertices/.style={draw, fill=black, circle, inner sep=0.5pt}]

\useasboundingbox (-0.3,-.35) rectangle (0.8,0.8) ;

\filldraw[fill=black!05, draw=black] (0.8,0)--(.1,0.8)--(-0.3,0)--(0.2,-0.4)--cycle;

\node[vertices] (a) at (0.8,0) {};
\node[vertices] (b) at (.1,0.8) {};
\node[vertices] (c) at (-0.3,0) {};
\node[vertices] (d) at (.2,-.4) {};

\foreach \to/\from in {a/b,a/d,b/c,b/d,c/d}
	\draw [-] (\to)--(\from);

\foreach \to/\from in {a/c}
	\draw [opacity=0.5, dashed] (\to)--(\from);
\end{tikzpicture}
\qquad
\begin{tikzpicture}
[scale=2.0, vertices/.style={draw, fill=black, circle, inner sep=0.5pt}]

\filldraw[fill=black!05, draw=black] (0,0)--(1,0)--(1.4,0.2)--(1.07,1.2)--(.67,1)--cycle;

\node[vertices] (a) at (0,0) {};
\node[vertices] (b) at (1,0) {};
\node[vertices] (c) at (1.4,0.2) {};
\node[vertices] (d) at (1.07,1.2) {};
\node[vertices] (e) at (.67,1) {};
\node[vertices] (f) at (.4,.2) {};

\foreach \to/\from in {a/b,b/c,c/d,d/e,e/a,b/e}
	\draw [-] (\to)--(\from);

\foreach \to/\from in {f/a,f/c,f/d}
	\draw [opacity=0.5, dashed] (\to)--(\from);

\end{tikzpicture}
\qquad
\begin{tikzpicture}
[scale=2.6, vertices/.style={draw, fill=black, circle, inner sep=0.5pt}]
\useasboundingbox (0,-.086) rectangle (1.35,0.92) ;
\filldraw[fill=black!05, draw=black] (0,0)--(1,0)--(1.25,.25)--(1.35,.75)--(.35,.75)--(.1,.5)--cycle;

\node[vertices] (a) at (0,0) {};
\node[vertices] (b) at (1,0) {};
\node[vertices] (c) at (.25,.25) {};
\node[vertices] (d) at (1.25,.25) {};
\node[vertices] (e) at (.1,.5) {};
\node[vertices] (f) at (1.1,.5) {};
\node[vertices] (g) at (.35,.75) {};
\node[vertices] (h) at (1.35,.75) {};

\foreach \to/\from in {a/b,b/d,d/h,h/g,g/e,e/a,e/f,f/h,f/b}
	\draw [-] (\to)--(\from);

\foreach \to/\from in {a/c,c/d,c/g}
	\draw [opacity=0.5, dashed] (\to)--(\from);
\end{tikzpicture}
\caption{Examples of $(n,d)$-parallelotopes in $\R^3$. From left to right: The $(3,1)$-parallelotope (the 3-simplex), the $(3,2)$-parallelotope, and the $(3,3)$-parallelotope (the 3-parallelotope).
The $(3,2)$-parallelotope is the only simple ECW polytope in $\R^3$ that is not a simplex or parallelotope.}\label{fig:polytope_example}
\end{center}
\end{figure}
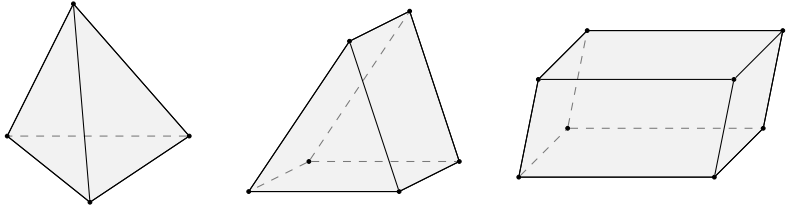

We comment on the two restrictions in our characterization: we only treat bounded and simple polytopes. Since circuit walks are defined to start and terminate at a vertex, a circuit walk in an unbounded polyhedron will never use a circuit direction from the recession cone. However, the facets of the recession cone may still intersect with facets from the ``bounded part" of the polyhedron to introduce new circuits to be used in a circuit walk.

Additionally, the analysis of the inner cones in a degenerate ECW polyhedron becomes significantly more challenging. Opposite inner cones need not completely mirror each other as in \Cref{thm:opposite_inner_cones} -- instead the edge directions of one inner cone could be a superset or subset of the directions of its opposite. Some of our results readily transfer to degenerate polytopes -- for example, all two-dimensional faces of (degenerate) ECW polytopes must be either triangles or parallelograms. Despite this restrictive property, such polytopes do appear in practice. As stated in \Cref{thm:examples}, one example is the (highly) degenerate fixed-size partition polytope (\Cref{sec:fixedpp}). 

An extension of our characterization of simple ECW polytopes to all ECW polyhedra is a natural but challenging next step for the presented line of research.

\section{Proofs for \Cref{thm:hierarchy,thm:reversible} }\label{sec:proofsforthms}

\subsection{Proof of \Cref{thm:hierarchy}}\label{sec:proof_thm_hierarchy}

We prove \Cref{thm:hierarchy}, which states that the levels in the hierarchy for integral polyhedra based on circuit walk behavior -- as depicted in \Cref{fig:hierarchy_reversible} -- are distinct. Although the result also follows from the more detailed \Cref{thm:examples}, we provide a short proof here using low-dimensional examples.

\begin{proof}[of \Cref{thm:hierarchy}]
Recall the two-dimensional integral polyhedra from \Cref{fig:examples}, and assume a minimal algebraic representation for each polyhedron. This implies that we can characterize the set of circuits for each polyhedron based on its geometric properties; in particular, for two-dimensional polyhedra, all circuits appear as edge directions. It follows that the set of circuits for the examples in \Cref{fig:example_a,fig:example_b,fig:example_c} is $\{\pm(1, 0), \pm(0, 1), \pm(1, 1), \pm(1, -1) \}$. 

Consider first \Cref{fig:example_a}. Clearly, taking a maximal step in the circuit direction $(1, -1)$ from the point $\vey^{(0)}$ leads to the non-integral point $\vey^{(1)}$. Therefore, the polyhedron is GCW but not ICW. 

In \Cref{fig:example_b}, it is not difficult to observe via complete enumeration that, starting at any vertex in the polyhedron and taking maximal steps along circuit directions, all reachable points are indeed integral. However, some of these points -- such as $\vey^{(1)}$, $\vey^{(2)}$, and $\vey^{(4)}$ -- need not be vertices. Hence the polyhedron is ICW but not VCW. 

Similarly, in \Cref{fig:example_c}, only vertices are reachable from any starting vertex via maximal circuit steps. However, some of these steps -- such as that from $\vey^{(0)}$ to $\vey^{(1)}$ -- need not travel along edges. Thus, the polyhedron is VCW but not ECW. 

Finally, assuming a minimal representation, the only circuits of the polyhedron in \Cref{fig:example_d} are $\{ \pm (1, 0) \pm (0, 1)\}$. It is then clear that, starting at any vertex, the only feasible circuit steps must travel along edges. Hence, the polyhedron is ECW. 
\eoproof  \end{proof}

\subsection{Proof of \Cref{thm:reversible}}\label{sec:reversible}

We prove \Cref{thm:reversible}, which states that all circuit walks in a polyhedron are reversible if and only if all of its circuit walks are vertex walks.
\begin{proof}[of \Cref{thm:reversible}]
First, let $P$ be a polyhedron in which all circuit walks are vertex walks and suppose there exists a circuit walk $\vev^{(1)} = \vey^{(0)}, \vey^{(1)},...,\vey^{(k)} = \vev^{(2)}$ in $P$ that is not reversible. Then the reversed walk $\vey^{(k)},...,\vey^{(1)}$ must not be maximal. In particular, let $i$ denote an index such that the step from $\vey{(i)}$ to $\vey^{(i-1)}$ is not maximal. Thus, taking a maximal step from $\vey^{(i)}$ along the direction $\vey^{(i-1)} - \vey^{(i)}$ leads to some point $\vez \neq \vey^{(i-1)}$ in $P$. However, this implies that $\vey^{(i-1)}$ can be expressed as a convex combination of $\vez$ and $\vey^{(i)}$, contradicting the fact that $\vey^{(i-1)}$ is a vertex of $P$.

Conversely, let $P$ be a polyhedron which contains a circuit walk $\vev^{(1)} = \vey^{(0)}, \vey^{(1)},...,\vey^{(k)} = \vev^{(2)}$ that is not a vertex walk. Then some step $\vey^{(i)}$ of the walk belongs to the strict interior of a face $F$ of $P$ with dimension greater than zero. Since $P$ is pointed, there exists a vertex $\veu$ in $F$. Furthermore, there exists a walk $\vey^{(i)}=\vez^{(0)},\vez^{(1)},...,\vez^{(t)}=\veu$ in $F$ which uses only edge directions of $F$ and takes maximal steps. Hence, $\vev^{(1)} = \vey^{(0)},...,\vey^{(i)}=\vez^{(0)},\vez^{(1)},...,\vez^{(t)}=\veu$ is a circuit walk in $P$. This circuit walk is not reversible since $\vey^{(i)}$ belongs to the strict interior of $F$ and the circuit direction immediately following $\vey^{(i)}$ on the walk is an edge direction of $F$.
\eoproof  \end{proof}

\section{Proof of \Cref{thm:tu_circuits}}\label{sec:proofofthm1}

The fact that TU polyhedra are ICW is not surprising since such a polyhedron has integral vertices for any integral right-hand side. We present a complete proof for a general integral polyhedron of the form $P = \{ \vex \in \R^n \colon A \vex = \veb, B \vex \leq \ved \}$. We note that this result can be obtained in various ways -- in particular, it can be achieved by solving a linear program over a certain one-dimensional TU polyhedron. However, the upcoming \Cref{lem:circuit_one-dim,lem:circuit_determinant} follow from basic linear algebra and have additional useful implications \cite{bv-18}. We also refer the reader to the intimately related and well-presented results surrounding Proposition 3.3 in \cite{o-10}, which apply to polyhedra in standard form.

We begin with an important property of circuits which is used in \cite{bv-18} and can be derived from Proposition~1 in \cite{kps-17}. A short proof is included for the sake of completeness.

\begin{lemma}[\cite{bv-18,kps-17}]\label{lem:circuit_one-dim}
Let $P = \{ \vex \in \R^n \colon A \vex = \veb, B \vex \leq \ved \}$ be a pointed polyhedron, let $\veg \in \ker(A)$ with coprime integer components be given, and let $B'$ denote the maximal row-submatrix of $B$ such that $B' \veg = \ve0$. Then $\veg$ is a circuit of $P$ if and only if $\rank \binom{A}{B'} = n-1$.
\end{lemma}
\begin{proof}
Suppose first that $\veg$ is a circuit of $P$. Since $P$ is pointed, $\rank \binom{A}{B'} \leq n-1$. If $\rank \binom{A}{B'} < n-1$, there exist rows of $B$ which can be added to $B'$ to form a row-submatrix $B''$ of $B$ with $\rank \binom{A}{B''} = n-1$. However, $\ker \binom{A}{B''}$ is then generated by some $\vey$ with $\supp(B \vey) \subsetneq \supp(B \veg)$, contradicting the fact that $\veg$ is a circuit.

Conversely, if $\rank \binom{A}{B'} = n-1$ then $\ker \binom{A}{B'}$ is generated by $\veg$. Hence, any $\vey \in \ker(A) \setminus \{\ve0\}$ satisfying  $\supp(B \vey) \subseteq \supp(B \veg)$ is a scalar multiple of $\veg$, implying that $\veg$ is a circuit.
\eoproof  \end{proof}

This property yields the following lemma, which generalizes Proposition~3.3 in \cite{o-10} by relating the components of a circuit $\veg$ and the corresponding vector $B \veg$ to the maximum absolute subdeterminant of the constraint matrices.

\begin{lemma}[\cite{o-10}]\label{lem:circuit_determinant}
Let $P = \{ \vex \in \R^n \colon A \vex = \veb, B \vex \leq \ved \}$ be a polyhedron and let $\veg \in \mathcal{C}(A,B)$ be a circuit of $P$. Then $ \max_i |\veg_i| \leq \Delta \left( M \right)$ and $ \max_i |(B \veg)_i| \leq \Delta \left( M \right)$, where $M := \binom{A}{B}$ and $\Delta(M)$ denotes the maximum absolute value of a subdeterminant of $M$. 
\end{lemma}

\begin{proof}
Let $B'$ denote the maximal row-submatrix of $B$ such that $B'\veg = \ve0$. By \Cref{lem:circuit_one-dim}, $\rank \binom{A}{B'} = n-1$. Consider the kernel of the matrix
\begin{align*}
    M_{B'} = \left( \begin{array}{c c}
A  &  \ve0\\
B'  & \ve0 \\
B & -I
\end{array}\right).
\end{align*}
Any $(\vex, \vey)^T \in \ker(M_{B'})$ must satisfy $\vey = B \vex$ in addition to $A\vex = \ve0$ and $B' \vex  = \ve0$. Since $\veg$ generates $\ker \binom{A}{B'}$, it follows that $\ker(M_{B'})$ is one-dimensional and generated by $(\veg, B\veg)^T$. Cramer's rule then yields
\begin{align*}
    \max_i \left\lvert (\veg, B \veg)^T_i \right\rvert \leq \Delta \left( \begin{array}{c c}
A  &  \ve0\\
B'  & \ve0 \\
B & -I
\end{array}\right) = \Delta(M).
\end{align*}
\eoproof  \end{proof}

As a corollary to \Cref{lem:circuit_determinant}, if $\binom{A}{B}$ is totally unimodular then any circuit $\veg \in \mathcal{C}(A,B)$ satisfies $\veg \in \{0,1,-1\}^n$ and $B\veg \in \{0,1,-1\}^m$. \Cref{thm:tu_circuits} now follows directly.

\begin{customthm}{3}
Let $P = \{ \vex \in \R^n \colon A \vex = \veb, B \vex \leq \ved \}$ be an integral polyhedron whose constraint matrix $\binom{A}{B}$ is totally unimodular. Then all circuit walks in $P$ are integral.
\end{customthm}

\begin{proof}%[Proof 1 of \Cref{thm:tu_circuits}]
We may assume that the right-hand side vectors $\veb$ and $\ved$ are integral, for otherwise either $P$ would be empty or the components of $\ved$ could be rounded up without changing the polyhedron. Let $\vex$ be any integral point in $P$ and let $\veg \in \mathcal{C}(A,B)$ be a feasible circuit direction at $\vex$. For any strict inequality $(B \vex)_i < \ved_i$ at $\vex$, we must have $(B \vex)_i \leq \ved_i-1$ since $B$ and $\ved$ are integral. Applying a step with step size $\alpha = 1$ in the circuit direction $\veg$, it follows by \Cref{lem:circuit_determinant} that $(B(\vex + \veg))_i = (B \vex)_i + (B \veg)_i \leq (B \vex)_i  + 1 \leq \ved_i$. Continuing in this direction until an inequality becomes strict, we see that the stopping point $\vex + \alpha \veg$ must be integral with integral step size $\alpha$.
\eoproof  \end{proof}

\section{Proof of \Cref{thm:examples}}\label{sec:proof_thm_examples}

In this section we prove \Cref{thm:examples}, which states that there exist well-known examples of integral polyhedra in each of the four layers of the hierarchy based on circuit walk behavior. Namely, we show that there are matroid polytopes which are GCW but not ICW (\Cref{sec:matroid}), that there are transporation polytopes which are ICW but not VCW (\Cref{sec:tp}), that there are bounded-size partition polytopes which are VCW but not ECW (\Cref{sec:bounded_size}), and that all fixed-size partition polytopes are ECW (\Cref{sec:fixedpp}). 

\subsection{Matroid Polytopes}\label{sec:matroid}

For a matroid $M$ with ground set $E$ and rank function $f$, the \textit{matroid polytope} $P(f)$ associated with $M$ can be represented as: 
\begin{equation*}
P(f) = \left\lbrace \vex \in \R^E \colon \vex \geq \ve0,  \ \sum_{e \in S} \vex_e \leq f(S) \ \forall S \subseteq E  \right\rbrace.
\end{equation*}
The vertices of $P(f)$ consist of all $\vex \in \{0,1\}^E$ that are incidence vectors of the independent sets of $M$. Two vertices are adjacent if and only if they differ in exactly one coordinate or they differ in exactly two coordinates while satisfying a certain quite technical ordering relation \cite{t-84}.

We examine the behavior of circuit walks in $P(f)$. The matrix $B$ in the system $B \vex \leq \ved$ defining $P(f)$ consists of the identity matrix and a submatrix corresponding to the constraints of the form $\sum_{e \in S} \vex_e \leq f(S)$. A vector $\veg \in \R^E \setminus \{\ve0\}$ with coprime integer components is a circuit of $P(f)$ when $B\veg$ is support-minimal over $\{B\vex \colon \vex \in \R^E \setminus \{\ve0 \} \}$. Note that each positive and negative unit vector is therefore a circuit of $P(f)$. Any circuit walk in $P(f)$ that exclusively uses these directions is an edge walk and can be interpreted as individually adding/removing elements from an independent set of $M$. Additionally, the difference of any two unit vectors is a circuit of $P(f)$. Combinatorially, these circuits correspond to swapping out some element of an independent set of $M$ for another. A circuit walk which exclusively uses these directions can traverse edges of $P(f)$ when the adjacency conditions of \cite{t-84} are satisfied, but it is possible that the walk is only a vertex walk and not an edge walk.

Furthermore, a circuit walk in $P(f)$ may be non-integral. Since $P(f)$ is a 0/1-polytope, only circuits from $\{0,1,-1\}^E$ can be used in an integral circuit walk. However, $P(f)$ has many circuits outside this domain. For instance, if $\{1,2,3,4\} \subseteq E$, consider the vector $\veg \in \R^E$ where:
\begin{equation*}
\veg_e =  \begin{cases}
2 & \text{if }e = 1 \\
-1 & \text{if } e \in \{2,3,4\} \\
0 & \text{otherwise}.
\end{cases}
\end{equation*}
To see that $\veg$ is a circuit, note that for any $S \subsetneq E$ containing $\{1\}$ and some $2$-set of $\{2,3,4\}$, we have $\sum_{e \in S} \veg_e = 0$. However, for any nonzero $\vey \in \R^E$ whose support is strictly contained in the support of $\veg$, it can be shown that $\sum_{e \in S} \vey_e \neq 0$ for some such subset $S$. Therefore, $\veg$ is indeed a circuit of $P(f)$. If $I := \{2, 3, 4\}$ is an independent set of $M$, then taking a step in the direction of $\veg$ starting at the corresponding vertex $\vev_I$ of $P(f)$ leads to a non-integral point in $P(f)$. In fact, for any rank function $f$, if $|E| \geq 4$ and if $M$ contains any independent set with size at least three, then a non-integral circuit walk can be constructed in $P(f)$.

\subsection{Transportation Polytopes}\label{sec:tp}

Given a set of $m$ suppliers and $n$ customers where $\veu \in \Z_+^m$ gives the supply of each supplier and $\vev \in \Z_+^n$ gives the demand of each customer, the corresponding \textit{transportation polytope} $P$ consists of all $\vey \in \R^{mn}$ that describe feasible commodity flow assignments from suppliers to customers:
\begin{align*}
\sum_{j=1}^n y_{ij} &= u_i, \ \ \ i=1,...,m \\
\sum_{i=1}^m y_{ij} &= v_j, \ \ \ j=1,...,n \\
y_{ij} & \geq 0, \ \  \  \ i = 1,...,m, \ j=1,...,n.
\end{align*}

Given a feasible flow assignment $\vey \in P$, the \textit{support graph} $B(\vey)$ of $\vey$ is the bipartite graph with partite sets corresponding to the suppliers and customers in which there exists an edge with weight $y_{ij}$ between supplier $i$ and customer $j$ if and only if $y_{ij} > 0$. It can be shown that any $\vey \in P$ is a vertex of $P$ if and only if $B(\vey)$ is acyclic. Further, two vertices are adjacent if and only if the union of the corresponding support graphs contains exactly one cycle \cite{kw-68}.

It follows that the circuits of $P$ consist of all simple cyclical exchanges of flow among the suppliers and customers \cite{bdfm-18}. Specifically, $\veg \in \R^{mn}$ is a circuit of $P$ if and only if the support of $\veg$ corresponds to a cycle in the complete bipartite graph $K_{m,n}$ whose edges alternately correspond to components $1$ and $-1$ of $\veg$. Hence, for any $\vey \in P$, applying a step in a feasible circuit direction from $\vey$ corresponds to reducing weight along every other edge of some cycle of $K_{m,n}$ while simultaneously increasing weight along the remaining edges of the cycle. The step terminates once the weight of an edge in $B(\vey)$ is reduced to 0.

Since the constraint matrix of $P$ is totally unimodular, we know by \Cref{thm:tu_circuits} that all circuit walks in $P$ are integral. However, it need not hold that all circuit walks are vertex walks. Equivalently, if $\vex$ is a vertex of $P$ and $\vey$ is reached by a maximal circuit step from $\vex$, then although $B(\vex)$ is acyclic, it need not hold that $B(\vey)$ is acyclic. 

\begin{figure}
\centering
\begin{subfigure}[t]{0.28 \textwidth}
\centering
\begin{tikzpicture}
[scale=1.2, vertices/.style={draw, fill=black, circle, inner sep=0.5pt}]

\node[vertices] (a) at (0,0) {};
\node[vertices] (b) at (0,-1) {};
\node[vertices] (c) at (0,-2) {};
\node[vertices] (d) at (1,0) {};
\node[vertices] (e) at (1,-1) {};
\node[vertices] (f) at (1,-2) {};

\node[left] at (a) {\small{$s_1$}};
\node[left] at (b) {\small{$s_2$}};
\node[left] at (c) {\small{$s_3$}};
\node[right] at (d) {\small{$c_1$}};
\node[right] at (e) {\small{$c_2$}};
\node[right] at (f) {\small{$c_3$}};

\foreach \to/\from in {a/e,b/d,b/f,c/e,c/f}
	\draw [-] (\to)--(\from);

\draw [fill, black] (a) circle [radius=0.05];
\draw [fill, black] (b) circle [radius=0.05];
\draw [fill, black] (c) circle [radius=0.05];
\draw [fill, black] (d) circle [radius=0.05];
\draw [fill, black] (e) circle [radius=0.05];
\draw [fill, black] (f) circle [radius=0.05];
\end{tikzpicture}
\caption{The acyclic support graph $B(\vex)$ for vertex $\vex$.}
\end{subfigure}
\qquad
\begin{subfigure}[t]{0.28 \textwidth}
\centering
\begin{tikzpicture}
[scale=1.2, vertices/.style={draw, fill=black, circle, inner sep=0.5pt}]

\node[vertices] (a) at (0,0) {};
\node[vertices] (b) at (0,-1) {};
\node[vertices] (c) at (0,-2) {};
\node[vertices] (d) at (1,0) {};
\node[vertices] (e) at (1,-1) {};
\node[vertices] (f) at (1,-2) {};

\node[left] at (a) {\small{$s_1$}};
\node[left] at (b) {\small{$s_2$}};
\node[left] at (c) {\small{$s_3$}};
\node[right] at (d) {\small{$c_1$}};
\node[right] at (e) {\small{$c_2$}};
\node[right] at (f) {\small{$c_3$}};

\foreach \to/\from in {a/e,b/d,b/f,c/e,c/f}
	\draw [-] (\to)--(\from);

\draw [fill, black] (a) circle [radius=0.05];
\draw [fill, black] (b) circle [radius=0.05];
\draw [fill, black] (c) circle [radius=0.05];
\draw [fill, black] (d) circle [radius=0.05];
\draw [fill, black] (e) circle [radius=0.05];
\draw [fill, black] (f) circle [radius=0.05];
	
\draw[draw=red, dashed, line width= 1,  ->] (a)--(d);
\draw[draw=red, dashed, line width= 1,  ->] (b)--(e);
\draw[draw=red, dashed, line width= 1,  ->] (d)--(b);
\draw[draw=red, dashed, line width= 1,  ->] (e)--(a);
\end{tikzpicture}
\caption{A maximal circuit step applied in the direction of $\veg$.}
\end{subfigure}
\qquad
\begin{subfigure}[t]{0.28 \textwidth}
\centering
\begin{tikzpicture}
[scale=1.2, vertices/.style={draw, fill=black, circle, inner sep=0.5pt}]

\node[vertices] (a) at (0,0) {};
\node[vertices] (b) at (0,-1) {};
\node[vertices] (c) at (0,-2) {};
\node[vertices] (d) at (1,0) {};
\node[vertices] (e) at (1,-1) {};
\node[vertices] (f) at (1,-2) {};

\node[left] at (a) {\small{$s_1$}};
\node[left] at (b) {\small{$s_2$}};
\node[left] at (c) {\small{$s_3$}};
\node[right] at (d) {\small{$c_1$}};
\node[right] at (e) {\small{$c_2$}};
\node[right] at (f) {\small{$c_3$}};

\foreach \to/\from in {a/d,b/e,b/f,c/e,c/f}
	\draw [-] (\to)--(\from);

\draw [fill, black] (a) circle [radius=0.05];
\draw [fill, black] (b) circle [radius=0.05];
\draw [fill, black] (c) circle [radius=0.05];
\draw [fill, black] (d) circle [radius=0.05];
\draw [fill, black] (e) circle [radius=0.05];
\draw [fill, black] (f) circle [radius=0.05];
\end{tikzpicture}
\caption{The cyclic support graph $B(\vex + \veg)$ of the resulting solution.}
\end{subfigure}
\caption{A circuit step in a transportation polytope that does not lead to a vertex.}\label{fig:transportation}
\end{figure}
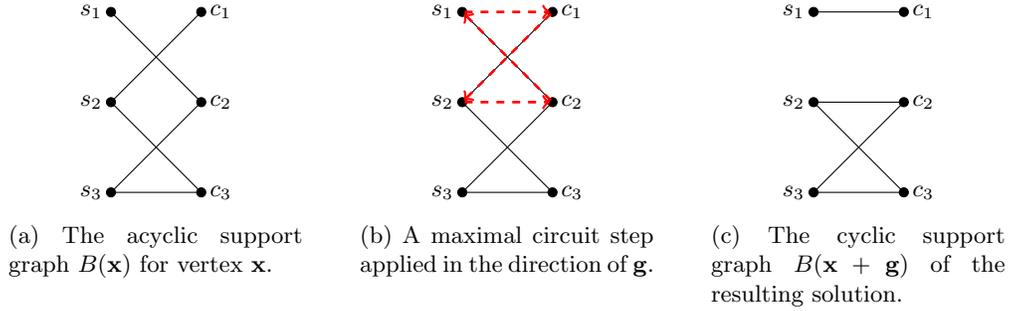

As an example, consider the transportation problem with suppliers $s_1, s_2, s_3$ and customers $c_1, c_2, c_3$, where $\veu = \vev = (1,2,2)^T$. Then $$\vex = (x_{11}, x_{12}, x_{13}, x_{21}, x_{22}, x_{23}, x_{31}, x_{32}, x_{33})^T = (0,1,0,1,0,1,0,1,1)^T$$ is a vertex of $P$ and $\veg = (1, -1, 0, -1, 1, 0, 0, 0, 0)^T$ is a circuit. However, after applying a step in the direction of $\veg$ from $\vex$ we reach a solution whose support graph is cyclic, and thus not a vertex. See \Cref{fig:transportation}.

\subsection{Bounded-size Partition Polytopes}\label{sec:bounded_size}

Similar to the partition polytopes for fixed-size clusterings introduced in \cite{b-13}, the \textit{bounded-size partition polytope} $PP(\kappa^{\pm})$ is associated with the partitioning of a set $X = \{x_1,...,x_n\}$ of items into clusters $C_1,...,C_k$, where each cluster $C_i$ must satisfy $\kappa_i^- \leq |C_i| \leq \kappa_i^+$ given $\kappa_i^-, \kappa_i^+ \in \Z_+$. For $i = 1,...,k$ and $j = 1,...,n$, let $y_{ij}$ be a binary variable indicating whether or not item $x_j$ is assigned to cluster $C_i$. Then $PP(\kappa^\pm)$ is the 0/1-polytope defined by the following system of constraints:
\begin{align*}
\sum_{i=1}^k y_{ij} &= 1 \ \ \ \ \  \  j=1,...,n \\
\sum_{j=1}^n y_{ij} &\geq \kappa_i^-  \ \ \ \ i=1,...,k\\
\sum_{j=1}^n y_{ij} &\leq \kappa_i^+  \ \ \ \  i=1,...,k \\
y_{ij} & \geq 0 \ \ \ \ \  \ i=1,...,k, \  j=1,...,n.
\end{align*}
We note that this polytope is an instance of the bounded-shape partition polytope described in \cite{bh-17} if $X$ is the standard basis of $\R^n$.

As the constraint matrix defining the polytope is totally unimodular and the right-hand sides are integral, the vertices of $PP(\kappa^\pm)$ consist of those $\vey \in \{0,1\}^{kn}$ corresponding to feasible clustering assignments. As in \cite{b-13} and \cite{bh-17}, given two such assignments $\vey^1, \vey^2$, define the \textit{clustering difference graph} $CDG(\vey^1, \vey^2)$ from $\vey^1$ to $\vey^2$ to be the directed graph with nodes $c_1,...,c_k$ where an edge $(c_i, c_\ell)$ with label $x_j$ is included if and only if $\vey^1_{i j} = \vey^2_{\ell j} = 1$ with $i \neq \ell$. Thus, the edges of $CDG(\vey^1, \vey^2)$ give all single-item transfers necessary in order to transform the clustering assignment of $\vey^1$ into that of $\vey^2$. An example of a simple $CDG$ is given in \Cref{fig:cdg_example}.

\begin{figure}[t]
    \centering
    \begin{tikzpicture}[vertices/.style={draw, fill=black, circle, inner sep=0pt, minimum size = 4pt, outer sep=0pt}, scale=1.7]

\node[vertices] (w_1) at (0,0.5) {};
\node[vertices] (w_2) at (0,1.5) {};
\node[vertices] (w_3) at (1, 1.5) {};
\node[vertices] (w_4) at (1, 0.5) {};

\node[vertices] (w_5) at (2.5,0.4) {};
\node[vertices] (w_6) at (2.5,1.6) {};

\foreach \to/\from in {w_1/w_2, w_2/w_3, w_3/w_4, w_4/w_1}
\draw[draw=black, line width= 1,  ->, >=latex]  (\to)--(\from);

\path [draw=black,line width= 1,->,>=latex,black] (w_5) edge[bend left=40] (w_6);
\path [draw=black,line width= 1,->,>=latex,black] (w_6) edge[bend left=40] (w_5);

\node[below left] at (w_1) {$c_1$};
\node[above left] at (w_2) {$c_2$};
\node[above right] at (w_3) {$c_3$};
\node[below right] at (w_4) {$c_4$};
\node[below] at (w_5) {$c_5$};
\node[above] at (w_6) {$c_6$};

\node (x_1) at (-0.15, 1.0) {$x_1$};
\node (x_2) at (0.5, 1.65) {$x_2$};
\node (x_3) at (1.15, 1.0) {$x_3$};
\node (x_4) at (0.5, 0.35) {$x_4$};

\node (x_5) at (2.1, 1.0) {$x_5$};
\node (x_6) at (2.9, 1.0) {$x_6$};

\end{tikzpicture}
    \caption{A depiction of the clustering difference graph $CDG(\vey^1, \vey^2)$ when $\vey^1$ corresponds to the clustering $C^1 = (\{x_1\},\{x_2\},\{x_3\},\{x_4\},\{x_5\},\{x_6\})$ and $\vey^2$ corresponds to $C^2 = (\{x_4\},\{x_1\},\{x_2\},\{x_3\},\{x_6\},\{x_5\})$. The edges of $CDG(\vey^1, \vey^2)$ describe the individual transfers needed to change $C^1$ into $C^2$.}
    \label{fig:cdg_example}
\end{figure}
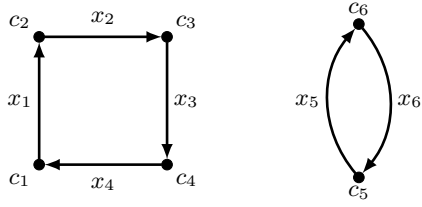

A necessary condition for vertices $\vey^1, \vey^2$ of $PP(\kappa^{\pm})$ to share an edge is that $CDG(\vey^1, \vey^2)$ is either a directed path or a directed cycle \cite{bh-17}. Here, we use clustering difference graphs to provide a full characterization of both the edges and the circuits of $PP(\kappa^\pm)$. 

To do so, given a pair of clustering assignments $\vey^1, \vey^2$, let $C^1 = (C^1_1,...,C^1_k)$ be the clustering of $X$ associated with $\vey^1$ and let $C^2 = (C^2_1,...,C^2_k)$ be the clustering associated with $\vey^2$. We say that a node $c_i$ of $CDG(\vey^1, \vey^2)$ is \textit{free} if $\kappa_i^- < |C^1_i| < \kappa_i^+$. Hence $C^1_i$ may receive or give away some item of $X$ and its size constraints will still be satisfied. Similarly, we say that $c_i$ is \textit{saturated} if $\kappa_i^- < |C^1_i| = \kappa_i^+$, that $c_i$ is \textit{depleted} if $\kappa_i^- = |C^1_i| < \kappa_i^+$, and that $c_i$ is \textit{fixed} if $\kappa_i^- = |C^1_i| = \kappa_i^+$.  Finally, we say that a directed graph $D$ is a \textit{valid $CDG$} at $\vey^1$ if there exists a vertex $\vey^i \in PP(\kappa^\pm)$ such that $D = CDG(\vey^1, \vey^i)$. The following lemma characterizes the edges of $PP(\kappa^\pm)$.

\begin{lemma}\label{lem:partition_polytope_edges}
Let $\vey^1, \vey^2$ be a pair of vertices in $PP(\kappa^\pm)$. Then $\vey^1$ and $\vey^2$ share an edge in $PP(\kappa^\pm)$ if and only if $CDG(\vey^1, \vey^2)$ cannot be non-trivially decomposed into valid $CDG$s at $\vey^1$. Equivalently, $\vey^1$ and $\vey^2$ share an edge if and only if $CDG(\vey^1, \vey^2)$ is a single edge, a single directed path in which no interior vertices are free, or a single directed cycle in which at most one vertex is free.
\end{lemma}

\begin{proof}
For the forward direction, assume that $\vey^1$ and $\vey^2$ are joined by an edge in $PP(\kappa^\pm)$ and suppose for the purpose of contradiction that $CDG(\vey^1, \vey^2)$ can be decomposed into valid $CDG$s at $\vey^1$: $CDG(\vey^1, \vey^3),...,CDG(\vey^1, \vey^\ell)$. Since $\vey^1$ and $\vey^2$ share an edge, there exists a vector $f \in \mathbb{R}^{kn}$ such that $f^T \vey^1 = f^T \vey^2$ but $f^T \vev > f^T \vey^1$ for any other vertex $\vev$ of $PP(\kappa^\pm)$. Specifically, it must hold that $f^T \vey^i > f^T \vey^1$ for $i = 3,...,\ell$.

However, note that the clustering of $\vey^2$ can be derived from that of $\vey^1$ by independent applications of the transfers given by each $CDG(\vey^1, \vey^i)$. In other words, we have $\vey^2 - \vey^1 = (\vey^3 - \vey^1) + \cdots + (\vey^\ell - \vey^1)$. Thus
\begin{equation*}
f^T(\vey^2 - \vey^1) = f^T(\vey^3 - \vey^1) + \cdots + f^T(\vey^\ell - \vey^1) > 0 + \cdots + 0 = 0,
\end{equation*}
a contradiction.

Conversely, assume that $CDG(\vey^1, \vey^2)$ cannot be decomposed in such a manner. See \Cref{fig:cdgs} for examples of possible structures for $CDG(\vey^1, \vey^2)$ and their implications. Note that if $CDG(\vey^1, \vey^2)$ were anything other than a single path or cycle, it could be decomposed into two valid $CDG$s at $\vey^1$ by isolating a directed cycle or a maximal path from the remainder of the graph. Further, if $CDG(\vey^1, \vey^2)$ were a directed path with a free interior vertex $c_i$, it could be decomposed into two valid $CDG$s at $\vey^1$ by splitting the graph at $c_i$ (\Cref{fig:cdg_a}). Similarly, if $CDG(\vey^1, \vey^2)$ were a directed cycle with at least two free vertices, we could split $CDG(\vey^1, \vey^2)$ into two directed paths at these vertices in order to form two valid $CDG$s at $\vey^1$ (\Cref{fig:cdg_c}). Hence, $CDG(\vey^1, \vey^2)$ must meet the specifications of the lemma (as in \Cref{fig:cdg_b,fig:cdg_d}).

\begin{figure}[t]
\centering
\begin{subfigure}[t]{0.48 \textwidth}
\centering
\begin{tikzpicture}
[scale=1.3, vertices/.style={draw, fill=black, circle, inner sep=0.5pt}]

\node[vertices] (a) at (0,0) {};
\node[vertices] (b) at (1,0) {};
\node[vertices] (c) at (2,0) {};
\node[vertices] (d) at (3,0) {};
\node[vertices] (e) at (4,0) {};

\node[above] at (a) {\footnotesize{$c_1$}};
\node[below] at (b) {\footnotesize{$c_2$}};
\node[above] at (c) {\footnotesize{$c_3$ (free)}};
\node[below] at (d) {\footnotesize{$c_4$}};
\node[above] at (e) {\footnotesize{$c_5$}};

\foreach \to/\from in {a/b,b/c}
	\draw[draw=blue, line width= 2,  ->]  (\to)--(\from);
	
\foreach \to/\from in {c/d,d/e}
	\draw[draw=green, line width= 2,  ->]  (\to)--(\from);

\draw [fill, black] (a) circle [radius=0.05];
\draw [fill, black] (b) circle [radius=0.05];
\draw [fill, black] (c) circle [radius=0.05];
\draw [fill, black] (d) circle [radius=0.05];
\draw [fill, black] (e) circle [radius=0.05];
\end{tikzpicture}
\caption{A case where $CDG(\vey^1, \vey^2)$ can be decomposed by splitting the directed path at the free vertex $c_3$. Hence, $\vey^1$ and $\vey^2$ do not share an edge in $PP(\kappa^\pm)$.}\label{fig:cdg_a}
\end{subfigure}
\quad
\begin{subfigure}[t]{0.48 \textwidth}
\centering
\begin{tikzpicture}
[scale=1.3, vertices/.style={draw, fill=black, circle, inner sep=0.5pt}]

\node[vertices] (a) at (0,0) {};
\node[vertices] (b) at (1,0) {};
\node[vertices] (c) at (2,0) {};
\node[vertices] (d) at (3,0) {};
\node[vertices] (e) at (4,0) {};

\node[above] at (a) {\footnotesize{$c_1$}};
\node[below] at (b) {\footnotesize{$c_2$ (fixed)}};
\node[above] at (c) {\footnotesize{$c_3$ (depleted)}};
\node[below] at (d) {\footnotesize{$c_4$ (saturated)}};
\node[above] at (e) {\footnotesize{$c_5$}};

\foreach \to/\from in {a/b,b/c}
	\draw[draw=blue, line width= 2,  ->]  (\to)--(\from);
	
\foreach \to/\from in {c/d,d/e}
	\draw[draw=red, line width= 2,  ->]  (\to)--(\from);

\draw [fill, black] (a) circle [radius=0.05];
\draw [fill, black] (b) circle [radius=0.05];
\draw [fill, black] (c) circle [radius=0.05];
\draw [fill, black] (d) circle [radius=0.05];
\draw [fill, black] (e) circle [radius=0.05];
\end{tikzpicture}
\caption{A case where $CDG(\vey^1, \vey^2)$ cannot be decomposed into valid $CDG$s -- any decomposition contains a directed $c_i,c_j$-path in which either $c_i$ is depleted/fixed or $c_j$ is saturated/fixed. Thus, $\vey^1$ and $\vey^2$ share an edge in $PP(\kappa^\pm)$.}\label{fig:cdg_b}
\end{subfigure}

\begin{subfigure}[t]{0.48 \textwidth}
\centering
\begin{tikzpicture}
[scale=1.4, vertices/.style={draw, fill=black, circle, inner sep=0.5pt}]

\node[vertices] (a) at (0,0) {};
\node[vertices] (b) at (-.3,1) {};
\node[vertices] (c) at (0.5,1.8) {};
\node[vertices] (d) at (1.3,1) {};
\node[vertices] (e) at (1,0) {};

\node[below left] at (a) {\footnotesize{$c_1$ (free)}};
\node[left] at (b) {\footnotesize{$c_2$  }};
\node[above] at (c) {\footnotesize{$c_3$ (free)}};
\node[right] at (d) {\footnotesize{$c_4$  }};
\node[below right] at (e) {\footnotesize{$c_5$}};

\foreach \to/\from in {a/b,b/c}
	\draw[draw=green, line width= 2,  ->]  (\to)--(\from);
	
\foreach \to/\from in {c/d,d/e,e/a}
	\draw[draw=blue, line width= 2,  ->]  (\to)--(\from);

\draw [fill, black] (a) circle [radius=0.05];
\draw [fill, black] (b) circle [radius=0.05];
\draw [fill, black] (c) circle [radius=0.05];
\draw [fill, black] (d) circle [radius=0.05];
\draw [fill, black] (e) circle [radius=0.05];
\end{tikzpicture}
\caption{A case where $CDG(\vey^1, \vey^2)$ can be decomposed by splitting the cycle at the free vertices $c_1$ and $c_3$. So $\vey^1$ and $\vey^2$ must not share an edge in $PP(\kappa^\pm)$.}\label{fig:cdg_c}
\end{subfigure}
\quad
\begin{subfigure}[t]{0.48 \textwidth}
\centering
\begin{tikzpicture}
[scale=1.4, vertices/.style={draw, fill=black, circle, inner sep=0.5pt}]

\node[vertices] (a) at (0,0) {};
\node[vertices] (b) at (-.3,1) {};
\node[vertices] (c) at (0.5,1.8) {};
\node[vertices] (d) at (1.3,1) {};
\node[vertices] (e) at (1,0) {};

\node[below left] at (a) {\footnotesize{$c_1$ (free)}};
\node[left] at (b) {\footnotesize{$c_2$ (fixed)  }};
\node[above] at (c) {\footnotesize{$c_3$ (saturated)}};
\node[right] at (d) {\footnotesize{$c_4$ (depleted)  }};
\node[below right] at (e) {\footnotesize{$c_5$ (fixed)}};

\foreach \to/\from in {a/b,b/c}
	\draw[draw=red, line width= 2,  ->]  (\to)--(\from);
	
\foreach \to/\from in {c/d,d/e,e/a}
	\draw[draw=blue, line width= 2,  ->]  (\to)--(\from);

\draw [fill, black] (a) circle [radius=0.05];
\draw [fill, black] (b) circle [radius=0.05];
\draw [fill, black] (c) circle [radius=0.05];
\draw [fill, black] (d) circle [radius=0.05];
\draw [fill, black] (e) circle [radius=0.05];
\end{tikzpicture}
\caption{A case where $CDG(\vey^1, \vey^2)$ cannot be decomposed into valid $CDG$s by the same reasoning as in \Cref{fig:cdg_b}. So $\vey^1$ and $\vey^2$ share an edge in $PP(\kappa^\pm)$.}\label{fig:cdg_d}
\end{subfigure}
\caption{Structures for $CDG(\vey^1, \vey^2)$ that can and cannot be decomposed into valid $CDG$s.}\label{fig:cdgs}
\end{figure}

To show that $\vey^1 = (\vey^1_{11},...,\vey^1_{kn})^T$ and $\vey^2 = (\vey^2_{11},...,\vey^2_{kn})^T$ are joined by an edge in $PP(\kappa^\pm)$, it suffices to find a vector $f \in \mathbb{R}^{kn}$ such that $f^T \vey^1 = f^T \vey^2 < f^T \vev$ for any other vertex $\vev \in PP(\kappa^\pm)$. Define such a vector $f = (f_{11},...,f_{kn})^T$ by setting
\begin{equation*}
f_{ij} = 
\begin{cases}
0 &\small{\text{if } x_j \in C^1_i \cap C^2_i }\\
0 &\small{\text{if $c_i$ is fixed, free, or an endpoint, and $x_j$ is incident to $c_i$ in $CDG(\vey^1, \vey^2)$ }} \\
-1  &\small{\text{if $c_i$ is a saturated non-endpoint, and $x_j$ is incident to $c_i$ in $CDG(\vey^1, \vey^2)$ }} \\
1  &\small{\text{if $c_i$ is a depleted non-endpoint, and $x_j$ is incident to $c_i$ in $CDG(\vey^1, \vey^2)$ }} \\
2n & \small{\text{otherwise. }}
\end{cases}
\end{equation*} 
Hence, if $\alpha$ denotes the number of non-endpoint saturated vertices in the path or cycle of $CDG(\vey^1, \vey^2)$ and $\beta$ denotes the number of non-endpoint depleted vertices in this component, we have $f^T \vey^1 = f^T \vey^2 = \beta - \alpha$.

Let $\vev$ be any other vertex of $PP(\kappa^\pm)$ and let $C^v = (C^v_1,...,C^v_k)$ denote its corresponding clustering. If $CDG(\vey^1, \vev)$ contains any edge $(c_i, c_j)$ not found in $CDG(\vey^1, \vey^2)$, then if $x_\ell$ is the label of this edge, we have $f_{j \ell} = 2n$ and $\vev_{j \ell} = 1$. Thus, since $\vev$ has exactly $n$ nonzero components, we obtain $f^T \vev > n \geq f^T \vey^1$.

Therefore, we may assume that $CDG(\vey^1, \vev)$ is a subgraph of $CDG(\vey^1, \vey^2)$.  Given that $\vev$ is not equal to $\vey^1$ or $\vey^2$, it follows that $CDG(\vey^1, \vev)$ is a nontrivial proper subgraph of $CDG(\vey^1, \vey^2)$ and hence must be a collection of disjoint directed paths. Assume first that $CDG(\vey^1, \vev)$ is a single directed $c_i, c_j$-path for some pair of vertices $c_i, c_j$. Then either $c_i$ is a saturated non-endpoint of $CDG(\vey^1, \vey^2)$ or $c_j$ is a depleted non-endpoint of $CDG(\vey^1, \vey^2)$. In the former case, note that since $|C^v_i| < |C^1_i|$, we have $f^T \vev \geq \beta - (\alpha - 1) > f^T \vey^1$. In the latter case, since $|C^v_j| > |C^1_j|$, we have $f^T \vev \geq (\beta + 1) - \alpha > f^T \vey^1$. 

If $CDG(\vey^1, \vev)$ is a collection of directed paths, then we can decompose it into valid clustering difference graphs $CDG(\vey^1, \vey^3),...,CDG(\vey^1, \vey^\ell)$ at $\vey^1$ where each consists of a single directed path. Then, as before, it holds that $f^T \vey^i > f^T \vey^1$ for $i=3,...,\ell$. This implies
\begin{equation*}
f^T(\vev - \vey^1) = f^T(\vey^3 - \vey^1) + \cdots + f^T(\vey^\ell - \vey^1) > 0 + \cdots +  0 = 0,
\end{equation*} 
as desired.
\eoproof  \end{proof}

The circuits of $PP(\kappa^\pm)$ have a significantly simpler characterization than the edges: Two vertices of $PP(\kappa^\pm)$ are joined by a circuit step if and only if the corresponding clustering difference graph is a directed path or cycle.

\begin{lemma}\label{lem:partition_polytope_circuits}
The circuits of $PP(\kappa^\pm)$ consist of those $\veg \in \{0,1,-1\}^{kn}$ that describe a single cyclical exchange or sequential movement of items among the clusters.
\end{lemma}
\begin{proof}
By \Cref{lem:circuit_determinant}, a circuit $\veg$ of $PP(\kappa^\pm)$ satisfies $\veg \in \{0,1,-1\}^{kn}$. Furthermore, we have $\sum_{i=1}^k \veg_{ij} = 0$ for $j=1,...,n$, implying that $\veg$ describes some set of transfers among a clustering of $X$ -- an entry $\veg_{ij} = 1$ implies that $x_j$ is added to cluster $C_i$, and $\veg_{ij} = -1$ implies that $x_j$ is removed from $C_i$. The inequality constraints $B \vey \leq \ved$ for this polytope consist of  non-negativity constraints and cluster size constraints. Hence, the support of any $B \veg$ consists of the support of $\veg$ along with the indices of any clusters whose sizes are changed by the transfers of $\veg$. It follows that $B \veg$ is support-minimal over all such transfers if and only if no subset of the cluster size changes implied by $\veg$ can be achieved by a proper subset of the transfers of $\veg$.

Suppose that $\veg$ describes a single cyclical exchange of items. Then no cluster sizes are changed when applying the transfers of $\veg$, but applying any nontrivial subset of these transfers results in at least two cluster size changes. Hence, $\veg$ is a circuit of $PP(\kappa^\pm)$. Similarly, if $\veg$ describes a single sequential movement of items among the clusters (i.e. the $CDG$ implied by the transfers of $\veg$ consists of a single directed path), then only two cluster sizes are changed by the transfers of $\veg$. Any nontrivial subset of these transfers necessarily changes the size of a third cluster. Hence, $\veg$ is again a circuit of $PP(\kappa^\pm)$.

Conversely, suppose that $\veg$ does not describe a single cyclical exchange or sequential movement, and let $D$ be a $CDG$ whose transfers are described by $\veg$. If $D$ contains a directed cycle as a subgraph, the transfers given by this cycle are a proper subset of those given by $\veg$ that result in no cluster size changes. Hence, $\veg$ must not be a circuit. If $D$ is acyclic, $D$ must properly contain a maximal directed path $D'$. The only two clusters whose sizes are changed as a result of the transfers of $D'$ correspond to the two endpoints of $D'$. By the maximality of $D'$, the sizes of these two clusters are also changed by the transfers of $\veg$. Hence, $\veg$ again must not be a circuit of $PP(\kappa^\pm)$.  
\eoproof  \end{proof}

By \Cref{cor:tu_0/1-polytopes}, we know that all circuit walks in $PP(\kappa^\pm)$ are vertex walks. However, \Cref{lem:partition_polytope_edges,lem:partition_polytope_circuits} highlight the differences between edges and circuits in $PP(\kappa^\pm)$ and show that its circuit walks can have much more general behavior than its edge walks.

\subsection{Fixed-size Partition Polytopes}\label{sec:fixedpp}

%Recall that the characterization in \Cref{thm:non-degerate_edges=circuits} is restricted to simple polytopes. Thus, ECW polytopes are not limited to $(n, d)$-parallelotopes when we allow for degeneracy. We close the discussion with a class of (highly) degenerate ECW polytopes that are relevant in practice. 

As a special class of transportation polytopes introduced in \cite{b-13}, the \textit{fixed-size partition polytope} $PP(\kappa)$ is associated with the partitioning of a set $X = \{x_1,...,x_n\}$ of items into clusters $C_1,...,C_k$ where each cluster $C_i$ has prescribed size $\kappa_i \in \Z_+$. Note that the well-known Birkhoff polytope is an instance of this polytope for $k=n$ and $\kappa_i = 1$ for $i=1,...,n$. Further, $PP(\kappa)$ is equivalent to the bounded-size partition polytope $PP(\kappa^\pm)$ of \Cref{sec:bounded_size} when $\kappa_i^- = \kappa_i^+$ for $i=1,...,k$. However, we show here that $PP(\kappa)$ has more restrictive circuit walk behavior than $PP(\kappa^\pm)$.

For $i=1,...,k$ and $j = 1,...,n$, let $y_{ij}$ be a binary variable indicating whether or not item $x_j$ is assigned to cluster $C_i$. Then $PP(\kappa)$ is the 0/1-polytope defined by the following system of constraints:
\begin{align*}
\sum_{j=1}^n y_{ij} &= \kappa_i \ \ \ i=1,...,k \\
\sum_{i=1}^k y_{ij} &= 1 \ \ \  \ j=1,...,n \\
y_{ij} & \geq 0 \ \ \ \ i=1,...,k, \ j=1,...n.
\end{align*}

As in \Cref{sec:bounded_size}, the vertices of $PP(\kappa)$ consist of all feasible clustering assignments $\vey \in \{0,1\}^{kn}$. Given two such assignments $\vey^1, \vey^2$, recall the definition of the \textit{clustering difference graph} $CDG(\vey^1, \vey^2)$ from $\vey^1$ to $\vey^2$. It is shown in \cite{b-13} that $\vey^1$ and $\vey^2$ are adjacent in $PP(\kappa)$ if and only if $CDG(\vey^1, \vey^2)$ consists of a single directed cycle. This characterization yields useful bounds on the combinatorial diameter of $PP(\kappa)$.

We now characterize the circuits of $PP(\kappa)$. Since it is an instance of the bounded-size partition polytope, a circuit $\veg$ of $PP(\kappa)$ must satisfy the condition of \Cref{lem:partition_polytope_circuits}: $\veg$ describes either a cyclical exchange or a sequential movement of items among the underlying clusters. However, $\veg$ must also satisfy $\sum_{j=1}^n \veg_{ij} = 0$ for $i=1,...,k$; that is, the set of transfers described by $\veg$ cannot change the size of any cluster. Therefore, $\veg$ is a circuit of $PP(\kappa)$ if and only if $\veg$ describes a cyclical exchange of items.

It follows that $PP(\kappa)$ is ECW. If $\vey^1$ is a vertex of $PP(\kappa)$ and $\veg$ is a feasible circuit direction at $\vey$, taking a maximal step in the direction of $\veg$ corresponds to applying a single cyclical exchange of items. Hence, the resulting vertex $\vey^2$ yields a clustering difference graph $CDG(\vey^1, \vey^2)$ consisting of a single directed cycle, implying that $\vey^1$ and $\vey^2$ indeed share an edge in $PP(\kappa)$.

\section{Proof of \Cref{thm:inner_and_elem_cones,thm:opposite_inner_cones,thm:non-degerate_edges=circuits}}\label{sec:proof_thms_ecw}

In this section we prove three characterizations of simple ECW polytopes. Recall the definitions of \textit{elementary arrangements}, \textit{elementary cones}, and \textit{inner cones} from \Cref{sec:edges=circuits}. Without loss of generality, consider a full-dimensional polyhedron $P = \{ \vex \in \R^n \colon B \vex \leq \ved \}$ given by a minimal representation. (Note that any polyhedron $P =  \{ \vex \in \R^{n'} \colon A \vex = \veb, B \vex \leq \ved \}$ can be expressed in this form with dimension $n=n' - \rank(A)$ by using the equality constraints to reduce the number of variables.) Recall that although the circuits of a polyhedron $P$ are determined by its constraint system,  a minimal representation ensures that each constraint appears as a facet in $P$, allowing us to characterize the set of circuits of $P$ via its geometric properties. Our first observation is that elementary cones are generated by circuits.

\begin{lemma}\label{lem:cone_circuits}
Let $P = \{\vex \in \R^n \colon B \vex \leq \ved\}$ be full-dimensional polyhedron. A pointed cone formed by the elementary arrangement of $P$ is generated by circuits of $P$. Conversely, each circuit of $P$ is the intersection of $n-1$ hyperplanes from the elementary arrangement.
\end{lemma}

\begin{proof}
Let $C$ be a cone from the elementary arrangement of $P$; i.e., $C$ can be represented as $C = \{ \vex \in \R^n \colon B' \vex \leq \ve0\}$, where $B'$ is a row-submatrix of $B$ (with some rows possibly scaled by -1). Since an extreme ray of $C$ is formed by the intersection of $n-1$ facets from the system $B' \vex \leq \ve0$, it follows from \Cref{lem:circuit_one-dim} that all such extreme rays are circuits of $P$.

Conversely, if $\veg$ is a circuit of $P$, \Cref{lem:circuit_one-dim} yields a row-submatrix $B'$ of $B$ with $\rank(B') = n-1$ such that $B' \veg = \ve0$. A set of $n-1$ linearly independent rows from $B'$ then corresponds to a set of hyperplanes from the elementary arrangement whose intersection contains $\veg$.
\eoproof  \end{proof}

Note that if $C = \{\vex \in \R^n \colon D \vex \leq \ve0 \}$ is an elementary cone of $P$ generated by circuits $\{\veg_1,...,\veg_k\}$, its opposite $-C = \{\vex \in \R^n \colon D \vex \geq \ve0 \}$ is also an elementary cone of $P$ generated by circuits $\{-\veg_1,...,-\veg_k\}$. \Cref{thm:inner_and_elem_cones} relates the inner cones of a simple ECW polytope to these elementary cones.

\begin{customthm}{5}[Elementary Cone Condition]
Let $P = \{ \vex \in \R^n \colon B \vex \leq \ved \}$ be a full-dimensional, simple polytope. All circuit walks in $P$ are edge walks if and only if for each vertex $\vev \in P$, the inner cone $I(\vev)$ is an elementary cone of $P$.
\end{customthm}

\begin{proof}
Suppose first $P$ is ECW and let $\vev$ be a vertex of $P$. The inner cone $I(\vev)$ is generated by the $n$ circuits of $P$ corresponding to the $n$ edge directions incident to $\vev$. If $I(\vev)$ is not an elementary cone, there exists an elementary cone $C$ of $P$ strictly contained in $I(\vev)$. By \Cref{lem:cone_circuits}, $C$ is generated by at least $n$ circuits of $P$. At least one of these circuits is different from the $n$ circuits generating $I(\vev)$. However, this implies that $I(\vev)$ contains a circuit that is not an edge direction incident to $\vev$. Since $P$ is bounded, this contradicts the fact that $P$ is ECW.

Conversely, suppose $P$ is not ECW. Then for some vertex $\vev \in P$, the inner cone $I(v)$ contains a circuit $\veg$ in addition to the $n$ circuits which generate $I(v)$. According to \Cref{lem:cone_circuits}, $\veg$ is generated by $n-1$ hyperplanes $H_1,...,H_{n-1}$ from the elementary arrangement of $P$. Suppose that none of these hyperplanes intersect the interior of $I(v)$ -- i.e., suppose that $I(v) \cap H_i$ is a face of $I(v)$ for $i=1,...,n-1$.  This implies that the intersection $I(v) \cap H_1 \cap \cdots \cap H_{n-1}$ is also a face of $I(v)$. However, this intersection must in fact be generated by $\veg$, contradicting the fact that $\veg$ is not an edge direction of $I(v)$. Therefore, one of the hyperplanes $H_1,...,H_{n-1}$ must intersect the interior of $I(v)$, implying that $I(v)$ is not an elementary cone.
\eoproof  \end{proof}

\Cref{thm:inner_and_elem_cones} is a quite straightforward characterization of simple ECW polytopes. We use it to prove a more descriptive characterization which we call the \textit{symmetric inner cone condition}. First, we show that this condition is a necessary property of any ECW polytope. Recall that given a pair of vertices $\veu, \vev$ of a polyhedron $P$, we let $P^{uv}$ denote the minimal face of $P$ containing $\veu$ and $\vev$ and let $I^{uv}(\veu), I^{uv}(\vev)$ denote the inner cones of $\veu, \vev$ with respect to $P^{uv}$. 

\begin{lemma}\label{lem:opposite_inner_cones}
Let $P = \{\vex \in \R^n \colon B \vex \leq \ved\}$ be a full-dimensional, simple polytope whose only circuit walks are edge walks. Then $I^{uv}(\veu) =  -I^{uv}(\vev)$ for all pairs of vertices $\veu, \vev$ in $P$.
\end{lemma}
\begin{proof}
Let $\veu, \vev$ be a pair of vertices in $P$, and let $d := \dim(P^{uv})$. If $d = 1$ and $\veu$ shares an edge with $\vev$ in $P$, the result is trivial, so assume that $d \geq 2$ and that $\veu, \vev$ are not adjacent. 

Note that the direction $\vev - \veu$ belongs to the inner cone $I^{uv}(\veu)$ and that its opposite $\veu - \vev$ belongs to $I^{uv}(\vev)$. By the definition of $P^{uv}$, both of these directions belong to the strict interiors of their respective cones, implying that the interior of $I^{uv}(\veu)$ intersects the interior of $-I^{uv}(\vev)$. By \Cref{thm:inner_and_elem_cones}, $I(\veu)$ and $I(\vev)$ are elementary cones of $P$. Hence, $I^{uv}(\veu)$ and $I^{uv}(\vev)$ are elementary cones of $P^{uv}$. Since their interiors intersect,  we must have $I^{uv}(\veu) = -I^{uv}(\vev)$.
\eoproof  \end{proof}

To show that the symmetric inner cone condition of \Cref{lem:opposite_inner_cones} is also a sufficient condition for a polytope to be ECW, first suppose that $P$ satisfies this condition while containing a pair of vertices sharing no facets. It then must hold that $P$ is a parallelotope, immediately implying that all of its circuit walks are edge walks.

\begin{lemma}\label{lem:parallelotope}
Let $P$ be an $n$-dimensional, simple polyhedron which satisfies $I^{uv}(\veu) =  -I^{uv}(\vev)$ for all pairs of vertices $\veu, \vev$ in $P$. If $P$ contains a pair of vertices that shares no facets, then $P$ is an $n$-parallelotope.
\end{lemma}

\begin{proof}
The statement is straightforward for $n \leq 2$, so assume $n \geq 3$. Let $\veu, \vev$ be a pair of vertices in $P$ that shares no facets. Hence, $I(\vev) = -I(\veu)$. Note that this immediately implies that $P$ is a polytope since any extreme ray of $P$ would have to belong to the two $n$-dimensional cones $I(\veu)$ and $-I(\veu)$. Furthermore, the $n$ facets containing $\veu$ are parallel to the $n$ facets containing $\vev$. Let $\mathcal{F}$ denote this set of $2n$ facets which form an $n$-parallelotope $Q$. Unless $P$ contains a facet outside of $\mathcal{F}$, we have $P = Q$. 

So suppose that $P$ contains facets outside of $\mathcal{F}$. Some such facet $F$ contains a vertex $\vew$ which is a neighbor of some vertex from $Q$. Such a vertex $\vew$ must not also be a vertex of $Q$ itself, else it would be a degenerate vertex in $P$. Hence, $\vew$ is formed by the intersection of $F$ with an edge $e$ of $Q$, cutting off some vertex $\vey \in Q$ from $P$. Note that if $e$ is incident to $\veu$, then $\vew$ shares no facets with $\vev$ and by assumption $I(\vew) = -I(\vev)$. However, this then implies $I(\vew) = I(\veu)$, a contradiction. Hence, $\vey$ must not be a neighbor of $\veu$. Similarly, $\vey$ must not be a neighbor of $\vev$.

Now, let $F \notin \mathcal{F}$ be such a facet in $P$ that cuts off a vertex $\vey \in Q$ sharing the most facets with $\veu$ in $Q$, and let $k$ denote the number of facets shared by $\vey$ and $\veu$ in $Q$. Since $\vey$ must not be adjacent to $\veu$ in $Q$, we have $k \leq n-2$. We will show that there then exist three vertices of $P$ from $Q$ which share more facets with $\veu$ than $\vey$ in $Q$ such that $\vey$ completes a 2-parallelotope with these three vertices. In order for $P$ to satisfy the symmetric inner cone condition, it must follow that $\vey \in P$, a contradiction.

In particular, let $F_1,...,F_n$ denote the facets of $Q$ incident to $\veu$ and $G_1,...,G_n$ denote the facets of $Q$ incident to $\vev$. Assume $F_1,...,F_{k},G_{k+1},...,G_{n}$ are the facets incident to $\vey$ in $Q$. As $Q$ is a parallelotope, the point $\vex_1$ formed by the intersection of facets $F_1,...,F_{k+1},G_{k+2},...,G_n$ is a vertex of $Q$. Since $\vex_1$ shares more than $k$ facets with $\veu$, it is not cut off by any facet outside of $\mathcal{F}$ and is also a vertex of $P$. Similarly, the point $\vex_2$ formed by $F_1,...,F_{k+2},G_{k+3},...,G_n$ and the point $\vex_3$ formed by $F_1,...,F_{k},G_{k+1},F_{k+2},G_{k+3},...,G_n$ are vertices of $P$.

However, consider the two-dimensional face $P'$ of $P$ formed by the intersection of facets $F_1,...,F_{k},G_{k+3},...,G_n$. It holds that $\vex_1, \vex_2, \vex_3$ are vertices of $P'$. Further, vertices $\vex_1$ and $\vex_3$ of $P'$ share no facets in $P'$, so in order for the symmetric inner cone condition to be satisfied, $P'$ has to be a 2-parallelotope. The vertex opposite of $\vex_2$ in $P'$ must then be $\vey$, but this contradicts $\vey \notin P$. Therefore, no such facet $F$ exists in $P$, implying that all facets of $P$ belong to $\mathcal{F}$ and thus $P = Q$. 
\eoproof  \end{proof}

Next, we prove that if a polytope $P$ satisfies the symmetric inner cone condition of \Cref{lem:opposite_inner_cones} but does not necessarily contain a pair of vertices sharing no facets as in \Cref{lem:parallelotope}, then $P$ must be a highly symmetric generalization of the parallelotope: the $(n,d)$-parallelotope. Recall \Cref{def:n_d-parallelotope} and the surrounding discussion in \Cref{sec:edges=circuits}. If $k$ denotes the minimum number of facets shared by any pair of vertices in $P$, we show that $P$ must be an $(n, n-k)$-parallelotope.

\begin{lemma}\label{lem:n-k_parallelotope}
Let $P$ be an $n$-dimensional, simple polytope that satisfies $I^{uv}(\veu) =  -I^{uv}(\vev)$ for all pairs of vertices $\veu, \vev$ in $P$. Then $P$ is an $(n, n-k)$-parallelotope, where $k$ is the minimum number of facets shared by any pair of vertices in $P$.
\end{lemma}

\begin{proof}
If $k = 0$, \Cref{lem:parallelotope} implies that $P$ is a parallelotope, so assume $k \geq 1$ and let $\veu, \vev$ be a pair of vertices in $P$ that shares exactly $k$ facets. Note that $P^{uv}$ is the intersection of the $k$ facets shared by $\veu$ and $\vev$. Hence, $P^{uv}$ is an $(n-k)$-dimensional polytope satisfying the symmetric inner cone condition in which $\veu$ and $\vev$ share no facets. By \Cref{lem:parallelotope}, $P^{uv}$ is an $(n-k)$-parallelotope. 

Next, note that $P$ has at least $2n - k$ facets: the $k$ facets shared by $\veu$ and $\vev$, the $n-k$ facets containing $\veu$ but not containing $\vev$, and the $n-k$ facets containing $\vev$ but not $\veu$. Let $\mathcal{F}$ denote this set of $2n-k$ facets. By the structure of $P^{uv}$, each vertex of $P^{uv}$ is the intersection of $n$ of these facets. In fact, we show that these are the only facets of $P$.

First, suppose that $P$ contains some facet $F \notin \mathcal{F}$ that intersects one of the edges of $P$ that leave $P^{uv}$. This intersection forms a vertex $\vew$ that shares an edge with some vertex $\veu' \in P^{uv}$. Since $P^{uv}$ is a parallelotope, there exists a vertex $\vev' \in P^{uv}$ that shares no facets with $\veu'$ in $P^{uv}$. Thus, the only facets shared by $\veu'$ and $\vev'$ in $P$ are the $k$ facets forming $P^{uv}$. However, since $\vev'$ must not be contained in $F$, this implies that $\vew$ and $\vev'$ share only $k-1$ facets in $P$, a contradiction with the choice of $k$.

Therefore, no facet outside of $\mathcal{F}$ intersects any of the edges leaving $P^{uv}$ in $P$. Hence, since $P$ is bounded, every edge that leaves $P^{uv}$ hits some facet of $\mathcal{F}$. Additionally, the vertex $\vew$ formed by this intersection shares exactly $k$ facets with some vertex of the $(n-k)$-parallelotope $P^{uv}$. Namely, if $\veu'$ again denotes the neighbor of $\vew$ in $P^{uv}$ and $\vev'$ is the vertex of $P^{uv}$ sharing exactly $k$ facets with $\veu'$ in $P$, then $\vew$ also shares exactly $k$ facets with $\vev'$ in $P$: $k-1$ of the facets forming $P^{uv}$ and one facet incident to $\vev'$ but not $\veu'$ in $P$. Thus, the face $P^{wv'}$ of $P$ is an $(n-k)$-parallelotope formed by the facets of $\mathcal{F}$ which contains the vertex $\vew$.

Proceeding inductively on combinatorial distance from $P^{uv}$, we see that all vertices of $P$ belong to some $(n-k)$-parallelotope face of $P$ formed by the facets of $\mathcal{F}$. Therefore, the only facets of $P$ are those of $\mathcal{F}$. It follows that $P$ is an $(n, n-k)$-parallelotope. 
\eoproof  \end{proof}

An $n$-parallelotope $P$ given by a minimal representation has only $n$ circuit directions: the directions of its $n$ classes of parallel edges. Hence, it is quite clear that $P$ is ECW. We show in the following lemma that this result generalizes to $(n,d)$-parallelotopes.

\begin{lemma}\label{lem:edge_walks}
All circuit walks in an $(n,d)$-parallelotope given by a minimal representation are edge walks.
\end{lemma}

\begin{proof}
Let $P$ be an $(n,d)$-parallelotope given by a minimal representation. By \Cref{thm:inner_and_elem_cones}, it suffices to show that the inner cone of each vertex in $P$ is an elementary cone. Thus, suppose for the purpose of contradiction that the inner cone of a vertex $\vev \in P$ is not an elementary cone. Then there exists a facet $F$ of $P$ with corresponding inequality $\veb^T \vex \leq \delta$ such that the parallel hyperplane $H_0 = \{ \vex \in \R^n \colon \veb^T \vex = \ve0\}$ divides the inner cone $I(\vev)$ into two $n$-dimensional cones. In particular, an edge of $I(\vev)$ leaves $H_0$ on either side of $H_0$. If this were not the case, the corresponding hyperplanes of all facets of $P$ would intersect $I(\vev)$ within a proper face of $I(\vev)$, implying that $I(\vev)$ is an elementary cone. 

Therefore, consider the hyperplane $H = \{ \vex \in \R^n \colon \veb^T \vex = \veb^T \vev \}$ which is parallel to $F$ and incident to $\vev$. Then there exist edges incident to $\vev$ in $P$ that leave $H$ on either side of $H$. Namely, since $P$ is simple, there exist facets $F_1$ and $F_2$ incident to $\vev$ such that the edge $e_1$ leaving $F_1$ at $\vev$ leads to a vertex $\vez_1$ satisfying $\veb^T \vez_1 < \veb^T \vev$, and the edge $e_2$ leaving $F_2$ at $\vev$ leads to a vertex $\vez_2$ satisfying $\veb^T \vez_2 > \veb^T \vev$. 

Since $P$ is an $(n,d)$-parallelotope, $\vev$ belongs to a $d$-parallelotope face of $P$ and there exists a vertex $\veu$ in this face sharing exactly $n-d$ facets with $\vev$. Further, since $\vev \notin F$ and $P$ has only $n+d$ facets, we must have $\veu \in F$.

Now suppose $\veu$ belongs to neither $F_{1}$ nor $F_{2}$. Then  the edges $e_1$ and $e_2$ must be two of the d edges incident to $\vev$ in $P^{uv}$ since $F_1$ and $F_2$ are not among the $n-d$ facets shared between $\veu$ and $\vev$. We will say that an edge with direction $\mathbf{e}$ is parallel to a facet with normal $\veb$ if and only if $\veb^T \mathbf{e} = 0$. Hence, neither of the edges $e_1$ and $e_2$ are parallel to $F$. However, in order to satisfy the symmetric inner cone condition $I^{uv} (\vev) = -I^{uv}(\veu)$, there must exist a pair of edges incident to $\veu$ in $P^{uv}$ which are parallel to $e_1$ and $e_2$. This contradicts the fact that since $P$ is simple, only one of the $n$ edges incident to $\veu$ is not parallel to $F$. 

Therefore, $\veu$ must belong to at least one of $F_1$ and $F_2$. If $\veu$ belongs to both $F_1$ and $F_2$, we may take a step along the edge incident to $\veu$ that leaves $F_{2}$ to reach a new vertex in $F_1$ sharing exactly $n-d$ facets with $\vev$. Hence, there always exists a vertex sharing $n-d$ facets with $\vev$ that is incident to precisely one of $F_1$ and $F_2$. Without loss of generality we will assume $\veu \in F_1$ and $\veu \notin F_2$. 

There exists an edge incident to $\veu$ that leaves $F_{1}$ and intersects some facet $G$ of $P$ at a vertex $\vew$. Recall that $\veu \in F$, so since the edge between $\veu$ and $\vew$ does not leave $F$, we also have $\vew \in F$. As in the proof of \Cref{lem:n-k_parallelotope}, $\vew$ shares exactly $n-d$ facets with $\vev$ in $P$. Hence, if $G$ is not $F_2$, it holds that $\vew$ is a vertex in $F$ belonging to neither $F_1$ nor $F_2$, and we again obtain the above contradiction. Thus, we may assume $G = F_2$ and therefore $\vew \in F_2$. 

Since $\vev$ shares exactly $n-d$ facets with $\vew$ in the $(n,d)$-parallelotope $P$, the face $P^{wv}$ is a $d$-parallelotope. Because $\vew \notin F_1$, the edge $e_1$ incident to $\vev$ is contained in $P^{wv}$. Furthermore, since $\vew \in F$ while $\vev \notin F$, it holds that $P^{wv} \cap F$ is a facet of $P^{uv}$. Note that in a parallelotope, the only edges that are not parallel to a given facet must intersect that facet. The edge $e_1$ is not parallel to $F$, so it must not be parallel to any of its lower dimensional faces. Thus, $e_1$ is not parallel to the facet $P^{wv} \cap F$ of $P^{wv}$, implying that $e_1$ intersects $P^{wv} \cap F$. This yields $\vez_1 \in F$. 

Similarly, if we consider the edge $e_2$ in the $d$-parallelotope $P^{uv}$, we see that $\vez_2 \in F$. However, this would imply that $\veb^T \vev > \veb^T \vez_1 = \delta =  \veb^T \vez_2 > \veb^T \vev$, a contradiction. 
\eoproof  \end{proof}

\Cref{thm:opposite_inner_cones,thm:non-degerate_edges=circuits} now follow directly from \Cref{lem:opposite_inner_cones,lem:n-k_parallelotope,lem:edge_walks}. If a simple polytope is ECW, then the symmetric inner cone condition must be satisfied, implying that the polytope is an $(n,d)$-parallelotope. Conversely, any simple polytope that satisfies the symmetric inner cone must be an $(n,d)$-parallelotope, which then implies the ECW property if the polytope is given by a minimal representation.

\section*{Acknowledgments}
Borgwardt gratefully acknowledges support through an ORS Large Grant at the University of Colorado Denver and the Collaboration Grant for Mathematicians ``Polyhedral Theory in Data Analytics" of the Simons Foundation.

\bibliography{literature}
\bibliographystyle{plain}

\end{document}